%% file: sanya-arxiv.tex
\documentclass[twoside]{amsart}

\newif\ifdraft
\draftfalse
\input{macros-bib.tex}

\mapcitekey{Saito:Theory}{Saito-th}
\mapcitekey{Saito:Introduction}{Saito-intro}
\mapcitekey{Saito:HodgeModules}{Saito-HM}
\mapcitekey{Saito:MixedHodgeModules}{Saito-MHM}
\mapcitekey{Saito:Kollar}{Saito-K}
\mapcitekey{Saito:Kaehler}{Saito-Kaehler}
\mapcitekey{Saito:ArithmeticMixed}{Saito-AMS}
\mapcitekey{Saito:NewDefinition}{Saito-new}
\mapcitekey{Saito:DModules}{Saito-D}
\mapcitekey{Saito:Monodromy}{Saito-M}
\mapcitekey{Saito:MixedHodgeComplexes}{Saito-MHC}
\mapcitekey{Saito:b-function}{Saito-b}

\mapcitekey{Beilinson:DerivedCategory}{Beilinson}
\mapcitekey{Galligo+Granger+Maisonobe:CroisementNormal}{GGM}
\mapcitekey{Sabbah:Bourbaki}{Sabbah-B}
\mapcitekey{Cattani+Kaplan+Schmid:L2}{CKS-L2}
\mapcitekey{Saito:ApplicationsHodgeModules}{Saito-app}
\mapcitekey{Brosnan+Fang+Nie+Pearlstein:Singularities}{BFNP}
\mapcitekey{Verdier:ExtensionPerverseSheaf}{Verdier}
\mapcitekey{Beilinson+Bernstein+Deligne}{BBD}
\mapcitekey{Mochizuki:wild}{Mochizuki-wild}
\mapcitekey{Sabbah:HodgeTheorySingularities}{Sabbah}
\mapcitekey{Kashiwara:conjecture}{Kashiwara-conj}
\mapcitekey{Simpson:twistor}{Simpson-tw}
\mapcitekey{Drinfeld:Kashiwara}{Drinfeld}
\mapcitekey{Gaitsgory:deJong}{Gaitsgory}
\mapcitekey{Mebkhout:AutreEquivalence}{Mebkhout}
\mapcitekey{Kashiwara:RiemannHilbert}{Kashiwara-RH}
\mapcitekey{Schmid+Vilonen:Unitary}{SV}
\mapcitekey{Burgos+Wildeshaus:Shimura}{BW}
\mapcitekey{Shimizu:HodgeModules}{Shimizu}
\mapcitekey{Peters+Steenbrink:MHS}{PS}
\mapcitekey{Budur+Mustata+Saito:BernsteinSato}{BMS}
\mapcitekey{Dimca+Maisonobe+Saito:Spectrum}{DMS}
\mapcitekey{Hotta+Takeuchi+Tanisaki:Book}{HTT}
\mapcitekey{Brav+Bussi+Dupont+Joyce+Szendroi:VanishingCycles}{BBDJS}
\mapcitekey{Behrend+Bryan+Szendroi:MotivicDT}{BBS}
\mapcitekey{Beilinson+Ginzburg+Soergel:Koszul}{BGS}
\mapcitekey{Popa+Schnell:OneForms}{one-forms}
\mapcitekey{Deligne:HodgeI}{Deligne-I}
\mapcitekey{Dettweiler+Sabbah:MiddleConvolution}{DS}
\mapcitekey{Kollar:DualizingI}{KollarI}
\mapcitekey{Kollar:DualizingII}{KollarII}
\mapcitekey{Kawamata:abelian}{Kawamata}
\mapcitekey{Fujita:FiberSpaces}{Fujita}
\mapcitekey{deCataldo+Migliorini:HodgeTheoryMaps}{dCM}
\mapcitekey{Deligne:Lefschetz}{Deligne-Lef}
\mapcitekey{Deligne:Finitude}{Deligne-fin}
\mapcitekey{Illusie:MonodromieLocale}{Illusie-mon}
\mapcitekey{Kashiwara:VanishingCycleSheaves}{Kashiwara-V}
\mapcitekey{Kashiwara+Kawai:Poincare}{KK}
\mapcitekey{Cattani+Kaplan+Schmid:Degeneration}{CKS}
\mapcitekey{Schmid:VHS}{Schmid}
\mapcitekey{Zucker:DegeneratingCoefficients}{Zucker}
\mapcitekey{Steenbrink+Zucker:VMHS}{SZ}
\mapcitekey{Kashiwara:Study}{Kashiwara-st}
\mapcitekey{Ramanujam:KodairaVanishing}{Ramanujam}
\mapcitekey{Saito:Induced}{Saito-ind}
\mapcitekey{Laumon:DF}{Laumon}
\mapcitekey{Laumon:TransformationsCanoniques}{Laumon-for}
\mapcitekey{Schnell:mhmduality}{mhmduality}
\mapcitekey{Kashiwara:Thesis}{Kashiwara-th}
\mapcitekey{Brylinski:TransformationsCanoniques}{Brylinski}
\mapcitekey{Schnell:Residues}{residues}
\mapcitekey{Cautis:OpenEmbedding}{Cautis}
\mapcitekey{Schnell:saito-vanishing}{saito-vanishing}
\mapcitekey{Popa:SaitoVanishing}{Popa}
\mapcitekey{Carlson+Griffiths:Infinitesimal}{CG}
\mapcitekey{Beilinson:HowToGlue}{Beilinson}
\mapcitekey{RIMS:HodgeTheory}{RIMS}

\newcommand{\HM}[2]{\operatorname{HM}(#1, #2)}
\newcommand{\HMs}[3]{\operatorname{HM}_{\leq #1}(#2, #3)}
\newcommand{\HMZ}[3]{\operatorname{HM}_{#1}(#2, #3)}
\newcommand{\HMp}[2]{\operatorname{HM}^p(#1, #2)}

\renewcommand{\MHM}[1]{\operatorname{MHM}(#1)}
\newcommand{\MHW}[1]{\operatorname{MHW}(#1)}
\newcommand{\MHMp}[1]{\operatorname{MHM}^p(#1)}
\newcommand{\MHMpex}[2]{\operatorname{MHM}_{\mathit{ex}}^p(#1, #2)}
\newcommand{\MHWp}[1]{\operatorname{MHW}^p(#1)}

\newcommand{\MHMalg}[1]{\operatorname{MHM}_{\mathrm{alg}}(#1)}
\newcommand{\MHMalgg}[1]{\operatorname{MHM}_{\mathrm{alg}} \bigl( #1 \bigr)}

\newcommand{\PervQ}{\operatorname{Perv}_{\QQ}}

\renewcommand{\shV}{\mathcal{V}}
\newcommand{\shVt}{\tilde{\shV}}

\newcommand{\Xt}{\tilde{X}}
\newcommand{\Xb}{\bar{X}}
\newcommand{\Xan}{X^{\mathit{an}}}
\newcommand{\Uan}{U^{\mathit{an}}}
\newcommand{\Xban}{\bar{X}^{\mathit{an}}}

\newcommand{\Zt}{\tilde{Z}}

\newcommand{\ellt}{\tilde{\ell}}
\newcommand{\Mt}{\tilde{M}}
\newcommand{\Kt}{\tilde{K}}
\newcommand{\ft}{\tilde{f}}
\newcommand{\ftp}{\ft_{+}}
\newcommand{\ftl}{\ft_{\ast}}
\newcommand{\fls}{f_{!}}
\newcommand{\kl}{k_{\ast}}

\newcommand{\ip}{i_{+}}
\newcommand{\pus}{p^{!}}

\newcommand{\ppsi}{ { }^{p} \psi}
\newcommand{\pphi}{ { }^{p} \phi}
\newcommand{\pshH}{{}^p \! R}
\newcommand{\omYX}{\omega_{Y/X}}

\newcommand{\HMZp}[3]{\operatorname{HM}_{#1}^p(#2, #3)}
\renewcommand{\shV}{\mathcal{V}}
\renewcommand{\DD}{\mathbf{D}}
\renewcommand{\derR}{\mathbf{R}}

\renewcommand{\derL}{\mathbf{L}}
\newcommand{\Rmod}{\mathscr{R}}
\newcommand{\Amod}{\mathscr{A}}
\newcommand{\Kmod}{\mathscr{K}}
\newcommand{\opp}{\mathit{opp}}
\newcommand{\shC}{\mathscr{G}}

\newcommand{\Dbcoh}{\mathrm{D}_{\mathit{coh}}^{\mathit{b}}}

\newcommand{\DbcohG}{\mathrm{D}_{\mathit{coh}}^{\mathit{b}}\mathrm{G}}

\newcommand{\DbcohF}{\mathrm{D}_{\mathit{coh}}^{\mathit{b}}\mathrm{F}}

\newcommand{\DF}{\mathrm{DF}}
\newcommand{\Db}{\mathrm{D}^{\mathit{b}}}

\newcommand{\Dbc}{\Db_{\mathit{c}}}

\DeclareMathOperator{\MF}{MF}

\DeclareMathOperator{\MG}{MG}
\newcommand{\MGcoh}{\MG_{\mathit{coh}}}
\newcommand{\Mcoh}{M_{\mathit{coh}}}

\newcommand{\Mmodf}{\Mmod_f}
\newcommand{\Mmodfal}[1]{\Mmod_{f,#1}}
\newcommand{\Mmodt}{\tilde{\Mmod}}
\newcommand{\CCst}{\CC^{\times}}
\newcommand{\Pell}{P_{\ell}}
\newcommand{\Pellt}{P_{\ellt}}
\DeclareMathOperator{\var}{var}

\newcommand{\omX}{\omega_X}

\newcommand{\VKM}{\mathcal{V}}
\newcommand{\famY}{\mathscr{Y}}

\begin{document}

\title[An overview of mixed Hodge modules]%
	{An overview of Morihiko Saito's theory\\ of mixed Hodge modules}
\author{Christian Schnell}
\address{%
Department of Mathematics\\
Stony Brook University\\
Stony Brook, NY 11794-3651}
\email{cschnell@math.sunysb.edu}

\begin{abstract}
After explaining the definition of pure and mixed Hodge modules on complex manifolds, we
describe some of Saito's most important results and their proofs, and then discuss
two simple applications of the theory.  
\end{abstract}
\date{\today}
\maketitle

\section{Introduction}

\subsection{Nature of Saito's theory}

Hodge theory on complex manifolds bears a surprising likeness to $\ell$-adic
cohomology theory on algebraic varieties defined over finite fields. This was pointed
out by Deligne \cite{Deligne-I}, who also proposed a ``heuristic dictionary''
for translating results from one language into the other, and used it to predict,
among other things, the existence of limit mixed Hodge structures and the notion of 
admissibility for variations of mixed Hodge structure. But the most important example of a
successful translation is without doubt Morihiko Saito's theory of mixed Hodge
modules: it is the analogue, in Hodge theory, of the mixed $\ell$-adic complexes
introduced by Beilinson, Bernstein, and Deligne \cite{BBD}. One of the main
accomplishments of Saito's theory is a Hodge-theoretic proof for the decomposition
theorem; the original argument in \cite[\S6.2]{BBD} was famously based on reduction
to positive characteristic.

Contained in two long papers \cite{Saito-HM,Saito-MHM}, Saito's theory is a vast
generalization of classical Hodge theory, built on the foundations laid by many
people during the 1970s and 1980s: the theory of perverse sheaves, $\Dmod$-module
theory, and the study of variations of mixed Hodge structure and their degenerations.
Roughly speaking, classical Hodge theory can deal with two situations: the cohomology
groups of a single complex algebraic variety, to which it assigns a mixed Hodge
structure; and the cohomology groups of a family of smooth projective varieties, to
which it assigns a variation of Hodge structure. The same formalism, based on the
theory of harmonic forms and the K\"ahler identities, also applies to cohomology
groups of the form $H^k(B, \shH)$, where $B$ is a smooth projective variety and
$\shH$ a polarizable variation of Hodge structure. Saito's theory, on the other hand,
can deal with arbitrary families of algebraic varieties, and with cohomology groups
of the form $H^k(B, \shH)$, where $B$ is non-compact or singular, and where $\shH$ is
an admissible variation of mixed Hodge structure. It also provides a nice formalism,
in terms of filtered $\Dmod$-modules, that treats all those situations in a
consistent way. 

Conceptually, there are two ways of thinking about mixed Hodge modules:
\begin{enumerate}
\item Mixed Hodge modules are \emph{perverse sheaves with mixed Hodge structure}.
They obey the same six-functor formalism as perverse sheaves; whenever an 
operation on perverse sheaves produces a $\QQ$-vector space, the corresponding
operation on mixed Hodge modules produces a mixed Hodge structure.
\item Mixed Hodge modules are \emph{a special class of filtered $\Dmod$-modules} 
with good properties. Whereas arbitrary filtered $\Dmod$-modules do not behave well
under various operations, the ones in Saito's theory do. The
same is true for the coherent sheaves that make up the filtration on the
$\Dmod$-module; many applications of the theory rely on this fact.
\end{enumerate}

The first point of view is more common among people working in Hodge theory;
the second one among people in other areas who are using mixed Hodge modules to
solve their problems. Saito's theory has been applied with great success in
representation theory \cite{BW,SV}, singularity theory
\cite{Saito:Steenbrink,BMS,BFNP,DMS}, algebraic geometry
\cite{Saito-K,Saito-AMS,one-forms}, and nowadays also in Donaldson-Thomas theory
\cite{BBDJS,BBS}. I should also mention that the whole theory has recently been generalized
to arbitrary semisimple representations of the fundamental group on algebraic varieties
\cite{Sabbah-B}.

\subsection{About this article}

The purpose of this article is to explain the definition of pure and mixed Hodge
modules on complex manifolds; to describe some of the most important results and
their proofs; and to discuss two simple applications of the theory. Along the way,
the reader will also find some remarks about mixed Hodge modules on analytic spaces
and on algebraic varieties, and about the construction of various functors.  While it
is easy to learn the formalism of mixed Hodge modules, especially for someone who is
already familiar with perverse sheaves, it helps to know something about the
internals of the theory, too. For that reason, more than half of the text is devoted
to explaining Saito's definition and the proofs of two crucial theorems; in return, I
had to neglect other important parts of the story.

The article has its origins in two lectures that I gave during the workshop on
geometric methods in representation theory in Sanya. In editing the notes, I also tried to incorporate
what I learned during a week-long workshop about mixed Hodge modules and their
applications, held in August 2013 at the Clay Mathematics Institute in Oxford. Both
in my lectures and in the text, I decided to present Saito's theory with a focus on
filtered $\Dmod$-modules and their properties; in my experience, this aspect of the
theory is the one that is most useful in applications.

For those looking for additional information about mixed Hodge modules, 
there are two very readable surveys by Saito himself: a short introduction that
focuses on ``how to use mixed Hodge modules'' \cite{Saito-intro}, and a longer
article with more details \cite{Saito-th}. The technical aspects of the theory are
discussed in a set of lecture notes by Shimizu \cite{RIMS}, and a brief axiomatic
treatment can be found in the book by Peters and Steenbrink \cite[Ch.~14]{PS}. There
are also unpublished lecture notes by Sabbah \cite{Sabbah}, who gives a
beautiful account of the theory on curves.

\subsection{Left and right D-modules}

An issue that always comes up in connection with Saito's theory is whether one
should use left or right $\Dmod$-modules. This is partly a matter of taste, because
the two notions are equivalent on complex manifolds, but each choice brings its
own advantages and disadvantages. Saito himself dealt with this problem by using left
$\Dmod$-modules in the introduction to \cite{Saito-HM,Saito-MHM}, and right
$\Dmod$-modules in the definitions and proofs. After much deliberation, I decided to
follow Saito and to write this article entirely in terms of \emph{right
$\Dmod$-modules}. Since this is sure to upset some readers, let me explain
why.

One reason is that the direct image and duality functors are naturally defined for
right $\Dmod$-modules. Both functors play important roles in the theory: the duality
functor is needed to define polarizations, and the direct image functor is needed for
example to define mixed Hodge modules on singular spaces. The idea is that
$\Dmod$-modules on a singular space $X$ are the same thing as $\Dmod$-modules on some
ambient complex manifold with support in $X$; to make this definition independent of
the choice of embedding, it is better to use right $\Dmod$-modules.  Another reason
is that it makes sense to distinguish, conceptually, between variations of Hodge
structure and Hodge modules: for example, it is easy to define inverse images for
variations of Hodge structure, but not for Hodge modules. Replacing a flat bundle by
the corresponding right $\Dmod$-module further emphasizes this distinction.

\subsection{Brief summary}

Let $X$ be a complex algebraic variety. Although Saito's theory works more generally
on complex manifolds and analytic spaces, we shall begin with the case of algebraic
varieties, because it is technically easier and some of the results are
stronger. Saito constructs the following two categories:%
\footnote{To be precise, $\HM{X}{w}$ should be $\HMp{X}{w}$, and $\MHM{X}$ should
be $\MHMalg{X}$.}
\begin{align*}
\HM{X}{w} &= \text{polarizable Hodge modules of weight $w$} \\
\MHM{X} &= \text{graded-polarizable mixed Hodge modules}
\end{align*}
Both are abelian categories, with an exact and faithful functor
\[
	\rat \colon \MHM{X} \to \PervQ(X)
\]
to the category of perverse sheaves (with coefficients in $\QQ$). All the familiar
operations on perverse sheaves -- such as direct and inverse images along arbitrary
morphisms, Verdier duality, nearby and vanishing cycles -- are lifted to mixed Hodge
modules in a way that is compatible with the functor ``$\rat$''. This means
that whenever some operation on perverse sheaves produces a $\QQ$-vector
space, the corresponding operation on mixed Hodge modules endows it with a mixed
Hodge structure. It also means that every semisimple perverse sheaf of geometric
origin is a direct summand of $\CC \tensor_{\QQ} \rat M$ for some polarizable Hodge
module $M$. After going to the derived category $\Db \MHM{X}$, one has a formalism of
weights, similar to the one for mixed complexes in \cite[\S5.1.5]{BBD}.

The relation between (mixed) Hodge modules and variations of (mixed) Hodge structure
is similar to the relation between perverse sheaves and local systems. 
Every polarizable variation of Hodge structure of weight $w - \dim Z$ on a
Zariski-open subset of an irreducible subvariety $Z \subseteq X$ extends uniquely to
a polarizable Hodge module $M \in \HM{X}{w}$ with strict support $Z$, which means
that $\rat M$ is the intersection complex of the underlying local system of
$\QQ$-vector spaces. Conversely, every polarizable Hodge module on $X$ is isomorphic
to a direct sum of objects of the above type, and this so-called \define{decomposition by
strict support} is unique. The category $\HM{X}{w}$ is semi-simple; the perverse
sheaves that one gets by applying the functor ``$\rat$'' are also semi-simple.

The category $\MHM{X}$ of graded-polarizable mixed Hodge modules is no longer
semi-simple, but every object $M \in \MHM{X}$ comes with a finite increasing
filtration $W_{\bullet} M$, called the \define{weight filtration}, such that 
\[
	\gr_{\ell}^W \! M = W_{\ell} M / W_{\ell-1} M \in \HM{X}{\ell}.
\]
Every mixed Hodge module is therefore an extension of several pure Hodge modules;
the extensions among these are not arbitrary, but are
controlled by an ``admissibility'' condition similar to the one appearing in the
theory of variations of mixed Hodge structure by Steenbrink and Zucker \cite{SZ} and
Kashiwara \cite{Kashiwara-st}. In the same way that a perverse sheaf can be obtained
by gluing together local systems on a stratification \cite{Verdier,Beilinson}, one
can think of a mixed Hodge module as being obtained by gluing together admissible
variations of mixed Hodge structure.

\subsection{Examples}

In order to get some intuition, let us look at a few typical examples of mixed Hodge
modules. The exact definitions will be given in \parref{par:def-HM} and
\parref{par:def-MHM} below; for the time being, we shall describe mixed Hodge
modules by giving the underlying perverse sheaf, the underlying filtered
$\Dmod$-module, and the weight filtration. In more detail, suppose that $M \in \MHM{X}$
is a mixed Hodge module on a smooth algebraic variety $X$, considered as a complex
manifold. The perverse sheaf is $\rat M$; its complexification is
isomorphic to the de Rham complex of a regular holonomic right $\Dmod_X$-module
$\Mmod$, and the additional piece of data is an increasing filtration $F_{\bullet}
\Mmod$ by coherent $\OX$-modules, compatible with the order filtration on $\Dmod_X$.

\begin{example}
The most basic example of a polarizable Hodge module on a nonsingular algebraic
variety $X$ of dimension $n$ is the canonical bundle $\omX$. It is naturally a
right $\Dmod$-module, and together with the filtration
\[
	F_{-n-1} \, \omX = 0 \quad \text{and} \quad F_{-n} \, \omX = \omX
\]
and the perverse sheaf $\QQ_X \decal{n}$, it is an object of the category
$\HM{X}{n}$, usually denoted by $\QQ_X^H \decal{n}$. One can obtain many additional
examples by applying various functors, such as the direct image by a morphism $f
\colon X \to Y$. When $f$ is proper, the filtration on $\shH^i \fl \QQ_X^H \decal{n}$
starts with the coherent $\OY$-module $R^i \fl \omX$; this fact is behind several results
in algebraic geometry, such as Koll\'ar's theorems about
higher direct images of dualizing sheaves (see \parref{par:Kollar}).
\end{example}

\begin{example}
More generally, one can consider an irreducible subvariety $Z \subseteq X$, and a
polarizable variation of Hodge structure of weight $w - \dim Z$ on a Zariski-open
subset of the smooth locus of $Z$. As mentioned above, it extends uniquely to
a polarizable Hodge module $M \in \HM{X}{w}$ with strict support $Z$. In this case,
the perverse sheaf $\rat M$ is simply the intersection complex of the underlying
local system of $\QQ$-vector spaces. The filtration $F_{\bullet} \Mmod$ on the
corresponding $\Dmod$-module is also uniquely determined by the Hodge filtration, but
hard to make explicit (except in special cases, for instance when the variation
is defined on the complement of a normal crossing divisor). 
\end{example}

\begin{example}
One case where one can describe the filtration in geometric terms is the family
of hyperplane sections of a smooth projective variety \cite{residues}. Suppose that
$X$ is a smooth projective variety of dimension $n$, and that $Y \subseteq X$ is a
nonsingular hyperplane section (in a fixed projective embedding). The entire
cohomology of $Y$ is determined by that of $X$, with the exception of the so-called
\define{variable part}
\[
	H_0^{n-1}(Y, \QQ) = 
		\ker \Bigl( H^{n-1}(Y, \QQ) \to H^{n+1} \bigl( X, \QQ(1) \bigr) \Bigr).
\]
Now let $B$ be the projective space parametrizing hyperplane sections of $X$, and let $f
\colon \famY \to B$ be the universal family. Let $B_0 \subseteq B$ be the open subset
corresponding to nonsingular hyperplane sections, and let $\shV$ be the variation of
Hodge structure on the variable part of their cohomology. As in the previous example,
it determines a Hodge module $M$ on $B$, and one can show that $M$ is isomorphic to a
direct summand of $\shH^0 \fl \QQ_{\famY}^H \decal{d+n-1}$, where $d = \dim B$. The
filtration $F_{\bullet} \Mmod$ on the underlying $\Dmod$-module can be
described using residues of meromorphic forms. For a nonsingular hyperplane
section $Y \subseteq X$, a classical theorem by Carlson and Griffiths \cite{CG} says
that the \define{residue mapping} 
\[
	\Res_Y \colon H^0 \bigl( X, \OmX^n(pY) \bigr) \to F^{n-p} H_0^{n-1}(Y, \CC)
\]
is surjective, provided the line bundle $\OX(Y)$ is sufficiently ample. Given an open
subset $U \subseteq B$ and a meromorphic form 
\[
	\omega \in H^0 \Bigl( U \times X, 
		\Omega_{B \times X}^{d+n} \bigl( (n+d+p) \famY \bigr) \Bigr),
\]
one can therefore take the residue along smooth fibers to obtain a holomorphic
section of $\omega_B \tensor \jl F^{-d-p} \shV$, where $j \colon B_0 \into B$ denotes
the inclusion. Then $F_p \Mmod$ is exactly the subsheaf of $\omega_B \tensor \jl
F^{-d-p} \shV$ given by all sections of this kind.
\end{example}

\begin{example}
One can obtain mixed Hodge modules by taking the direct image along an
open embedding $j \colon U \into X$. Suppose that $X$ is a nonsingular algebraic
variety of dimension $n$, and that $U = X \setminus D$ is the complement of a
divisor. Then $\jl \QQ_U^H \decal{n}$ is a mixed Hodge module on $X$; the underlying
perverse sheaf is $\derR \jl \QQ \decal{n}$, and the underlying regular holonomic
$\Dmod$-module is $\omX(\ast D)$, the sheaf of meromorphic $n$-forms on $X$ that are
holomorphic on $U$. When the divisor $D$ is nonsingular, one can also say explicitly
what the Hodge filtration and the weight filtration are. The Hodge filtration
satisfies
\[
	F_{-n-1} \, \omX(\ast D) = 0 \quad \text{and} \quad
	F_p \, \omX(\ast D) = \omX \bigl( (p + n + 1) D \bigr) 
		\quad \text{for $p \geq -n$,}
\]
and is therefore essentially the filtration by pole order; the weight filtration
satisfies
\[
	W_n \jl \QQ_U^H \decal{n} = \QQ_X^H \decal{n} 
		\quad \text{and} \quad
	W_{n+1} \jl \QQ_U^H \decal{n} = \jl \QQ_U^H \decal{n},
\]
and the quotient $\gr_{n+1}^W \jl \QQ_U^H \decal{n}$ is isomorphic to the direct
image of $\QQ_D^H(-1) \decal{n-1}$, with a Tate twist to change the weight from $n-1$
to $n+1$. When the divisor $D$ is singular, both filtrations still exist, but are
hard to describe concretely.
\end{example}

\begin{example}
An interesting special case of the previous construction has been studied by
Cautis and Kamnitzer \cite{Cautis}. Let $G(k,n)$ denote the Grassmannian parametrizing
$k$-dimensional subspaces of $\CC^n$, and consider the divisor
\[
	D = \menge{(V,W) \in G(k,n) \times G(n-k,n)}{V \cap W \neq 0}
\]
inside the product $X = G(k,n) \times G(n-k,n)$; without loss of generality, $k \leq
n/2$. We would like to describe the direct image $\jl \QQ_U^H \decal{d} \in \MHM{X}$,
where $d = 2k(n-k)$ is the dimension of $X$. The open set $U$ is one of the orbits of
the natural $\GL(n)$-action on $X$; the divisor $D$ is the union of the other orbits
\[
	Z_r = \menge{(V,W) \in G(k,n) \times G(n-k,n)}{\dim V \cap W \geq r}
\]
for $1 \leq r \leq k$, and therefore singular once $k \geq 2$. Using the group
action, one can show that for $r = 0, \dotsc, k$, the graded quotient
\[
	\gr_{d+r}^W \jl \QQ_U^H \decal{d} \in \HM{X}{d+r}
\]
of the weight filtration is the Hodge module corresponding to the constant variation
of Hodge structure $\QQ \bigl( -\frac{1}{2}r(r+1) \bigr)$ on the orbit $Z_r$; it has
the correct weight because
\[
	 \dim Z_r + r(r+1) = d - r^2 + r(r+1) = d + r.
\]
One can also compute the graded module $\gr_{\bullet}^F \omX(\ast D)$ and show that
the characteristic variety is the union of the conormal varieties of the orbit
closures.
\end{example}

\subsection{Acknowledgements}

I thank Wilfried Schmid for inviting me to Sanya, and the editors of the workshop
proceedings for giving me the time to rewrite my lecture notes. Morihiko Saito and
Claude Sabbah sent me several pages of detailed comments about the first version;
I am very grateful for their advice. I also take this opportunity to thank
the participants of the Clay workshop for many productive discussions about
Saito's theory. During the preparation of this paper, I have been supported in part by
grant DMS-1331641 from the National Science Foundation.

\section{Pure Hodge modules}

After a brief review of nearby and vanishing cycle functors, we explain the
definition of (polarizable) Hodge modules on complex manifolds. A Hodge module is
basically a certain type of $\Dmod$-module with a filtration and a rational
structure. The goal is understand the need for the various technical conditions that
Saito imposes on such objects in order to define the category $\HMp{X}{w}$ of
polarizable Hodge modules of weight $w$ on a complex manifold $X$.

\subsection{Basic objects}

Saito's theory is an extension of classical Hodge theory, and the objects that
he uses are a natural generalization of variations of Hodge structure. Recall that a
\define{polarized variation of Hodge structure of weight $w$} on a complex manifold
$X$ consists of the following four things: 
\begin{enumerate}
\item A local system $V$ of finite-dimensional $\QQ$-vector spaces.
\item A holomorphic vector bundle $\shV$ with integrable connection
$\nabla \colon \shV \to \OmX^1 \tensor \shV$.
\item A Hodge filtration $F^{\bullet} \shV$ by holomorphic subbundles.
\item A bilinear form $Q \colon V \tensor_{\QQ} V \to \QQ(-w)$, where $\QQ(-w) = (2
\pi i)^{-w} \QQ \subseteq \CC$.
\end{enumerate}
They are related by the condition that the local system of $\nabla$-flat holomorphic
sections of $\shV$ is isomorphic to $V \tensor_{\QQ} \CC$; in particular, $\shV
\simeq V \tensor_{\QQ} \OX$. At every point $x \in X$, the filtration $F^{\bullet}
\shV_x$ and the rational structure $V_x$ should define a Hodge structure of weight
$w$ on the $\CC$-vector space $\shV_x$, and this Hodge structure should be polarized
by the bilinear form $Q_x$. Globally, the Hodge filtration is required to satisfy the
Griffiths transversality relation
\[
	\nabla \bigl( F^p \shV \bigr) \subseteq \OmX^1 \tensor F^{p-1} \shV.
\]
By the holomorphic Poincar\'e lemma, the holomorphic de Rham complex
\[
	\shV \to \OmX^1 \tensor \shV \to \dotsb \to \OmX^{\dim X} \tensor \shV
\]
is a resolution of the locally constant sheaf $V \tensor_{\QQ} \CC$; this gives
another way to describe the relationship between $V$ and $(\shV, \nabla)$. Lastly,
recall that a variation of Hodge structure is called \define{polarizable} if it
admits at least one polarization.  

Saito's idea is to generalize variations of Hodge structure by allowing perverse
sheaves instead of local systems; because of the Riemann-Hilbert correspondence, it
then becomes necessary to use regular holonomic $\Dmod$-modules instead of vector
bundles with integrable connection. This leads to the following definition.

\begin{definition}
A \define{filtered regular holonomic $\Dmod$-module with $\QQ$-structure} is a triple
$M = (\Mmod, F_{\bullet} \Mmod, K)$, consisting of the following
objects:
\begin{enumerate}
\item A constructible complex of $\QQ$-vector spaces $K$.
\item A regular holonomic right $\Dmod_X$-module $\Mmod$ with an isomorphism
\[
	\DR(\Mmod) \simeq \CC \tensor_{\QQ} K.
\]
By the Riemann-Hilbert correspondence, this makes $K$ a perverse sheaf.
\item A good filtration $F_{\bullet} \Mmod$ by $\OX$-coherent subsheaves of $\Mmod$,
such that
\[
	F_p \Mmod \cdot F_k \Dmod \subseteq F_{p+k} \Mmod
\]
and such that $\gr_{\bullet}^F \! \Mmod$ is coherent over $\gr_{\bullet}^F \! \Dmod_X
\simeq \Sym^{\bullet} \shT_X$.
\end{enumerate}
\end{definition}

Here $\shT_X$ is the tangent sheaf of the complex manifold $X$. The \define{de
Rham complex} of a right $\Dmod$-module $\Mmod$ is the following complex:
\[
	\DR(\Mmod) = \Bigl\lbrack 
		\Mmod \tensor \textstyle\bigwedge\nolimits^n \shT_X \to
		\dotsb \to \Mmod \tensor \shT_X \to \Mmod
	\Bigr\rbrack \decal{n}
\]
It is concentrated in degrees $-n, \dotsc, -1, 0$, where $n = \dim X$; and naturally
filtered by the family of subcomplexes
\[
	F_p \DR(\Mmod) = \Bigl\lbrack 
		F_{p - n} \Mmod \tensor \textstyle\bigwedge\nolimits^n \shT_X \to
		\dotsb \to F_{p-1} \Mmod \tensor \shT_X \to F_p \Mmod
	\Bigr\rbrack \decal{n}.
\]
As I mentioned earlier, one can think of $M$ as a perverse sheaf with an
additional filtration on the corresponding regular holonomic $\Dmod$-module; the
functor ``$\rat$'' to the category of perverse sheaves is of course defined by
setting $\rat M = K$.  

\begin{example}
To every variation of Hodge structure, one can associate a filtered regular
holonomic $\Dmod$-module with $\QQ$-structure by the following procedure. The
perverse sheaf is $K = V \decal{n}$, and the regular holonomic $\Dmod$-module is
\[
	\Mmod = \omX \tensor_{\OX} \shV,
\]
with right $\Dmod$-module structure given by the rule $(\omega \tensor m) \cdot \xi =
(\omega \cdot \xi) \tensor m - \omega \tensor \nabla_{\xi} m$ for local holomorphic
sections $\omega \in \omX$, $m \in \Mmod$, and $\xi \in \shT_X$.  It is filtered by
the coherent subsheaves 
\[
	F_p \Mmod = \omX \tensor_{\OX} F^{-p-n} \shV;
\]
we will see later that the shift in the filtration gives $M$ the weight $w + n$. The
definitions are set up in such a way that the de Rham complex $\DR(\Mmod)$ is
isomorphic to the holomorphic de Rham complex of $\shV$, shifted $n$ steps to the
left; because of the holomorphic Poincar\'e lemma, this means that $\DR(\Mmod) \simeq
\CC \tensor_{\QQ} K$.
\end{example}

\begin{example}
A typical example is the constant variation of Hodge structure $\QQ_X$ of weight $0$.
The corresponding triple is $\bigl( \omX, F_{\bullet} \omX, \QQ_X \decal{n} \bigr)$; 
here $\omX$ is naturally a right $\Dmod$-module, and the filtration is such that
$\gr_p^F \omX = 0$ for $p \neq -n$.  
\end{example}

\begin{example}
The \define{Tate twist} of $M$ by an integer $k$ is the new triple
\[
	M(k) = \Bigl( \Mmod, F_{\bullet - k} \Mmod, K \tensor_{\QQ} \QQ(k) \Bigr),
\]
where $\QQ(k) = (2 \pi i)^k \QQ \subseteq \CC$. For variations of Hodge structure,
this definition specializes to the usual notion.
\end{example}

\begin{example}
When $X$ is compact, the cohomology groups $H^i \bigl( X, \DR(\Mmod) \bigr)$ are
finite-dimensional $\CC$-vector spaces; they come with a filtration induced
by $F_{\bullet} \DR(\Mmod)$, and a $\QQ$-structure induced by
the isomorphism with $H^i(X, K) \tensor_{\QQ} \CC$. We will see later that they are
polarizable Hodge structures when $X$ is projective and $M$ is a polarizable Hodge
module.
\end{example}

The class of all filtered regular holonomic $\Dmod$-modules with $\QQ$-structure is
of course much too large for the purposes of Hodge theory; the problem is to find a
subclass of ``Hodge modules'' that is small enough to have good properties, but still
large enough to contain all polarizable variations of Hodge structure. We would also
like the resulting theory to be similar to the theory of perverse sheaves: for
example, there should be a way of extending a polarizable variation of Hodge
structure on a Zariski-open subset of an irreducible subvariety $Z \subseteq X$ to a
polarizable Hodge module on $X$, similar to taking the intersection complex of a
local system.

We shall define polarizable Hodge modules by imposing several additional conditions
on $(\Mmod, F_{\bullet} \Mmod)$ and $K$; these conditions should be strong enough
that, at the end of the day, every polarizable Hodge module $M$ of weight $w$ is of
the form
\[	
	M \simeq \bigoplus_{Z \subseteq X} M_Z,
\]
where $M_Z$ is obtained from a polarizable variation of Hodge structure of
weight $w - \dim Z$ on a Zariski-open subset of the irreducible subvariety $Z$. This is
a reasonable demand for two reasons: (1) The perverse sheaves appearing in the
decomposition theorem are direct sums of intersection complexes. (2) The underlying
local system of a polarizable variation of Hodge structure is semisimple
\cite{Deligne-fin}.

In other words, we know from the beginning what objects we want in the theory; the
problem is to find a good description by conditions of a local nature. Moreover, the
conditions have to be such that we can check them after applying various operations.
Saito's solution is, roughly speaking, to consider only those objects
whose ``restriction'' to every point is a polarized Hodge structure; this idea does not
directly make sense at points where $\Mmod$ is singular, but can be implemented with
the help of the nearby and vanishing cycle functors (which are a replacement for the
naive operation of restricting to a divisor). 

\begin{note}
As Schmid pointed out after my lectures, the $\QQ$-structure is not essential. In
fact, Saito's theory works just as well when $K$ is a constructible complex of
$\RR$-vector spaces; with a little bit of extra effort, one can even get by with
coefficients in $\CC$. This point is discussed in more detail in \cite[Section~3.2]{DS}.
\end{note}

\subsection{Review of nearby and vanishing cycles}
\label{par:review-cycles}

Before we can understand Saito's definition, we need to become sufficiently familiar
with nearby and vanishing cycle functors, both for perverse sheaves and for regular
holonomic $\Dmod$-modules. This topic is somewhat technical -- but it is really at the
heart of Saito's theory, and so we shall spend some time reviewing it. I decided
to state many elementary properties in the form of exercises; working them out
carefully is strongly recommended.

To begin with, suppose that $X$ is a complex manifold, and that $f \colon X \to
\Delta$ is a holomorphic function that is submersive over the punctured unit disk
$\dst = \Delta \setminus \{0\}$.  Remembering that the function $e \colon \HH \to \Delta$,
$e(z) = e^{2 \pi i z}$, makes the upper half-plane into the universal covering space
of $\dst$, we get the following commutative diagram:
\[
\begin{tikzcd}
\Xt \dar \rar{k} & X \dar{f} & X_0 \dar \lar[swap]{i} \\ 
\HH \rar{e} & \Delta & \{0\} \lar
\end{tikzcd}
\]
Here $\Xt$ is the fiber product of $X$ and $\HH$ over $\Delta$. For a constructible
complex of $\CC$-vector spaces $K$ on $X$, we define the complex of \define{nearby cycles}
\[
	\psi_f K = i^{-1} \derR \kl \bigl( k^{-1} K \bigr);
\]
roughly speaking, this means that we pull back $K$ to the ``generic fiber'' of $f$,
and then retract onto the ``special fiber''. Accordingly, $\psi_f K$ contains 
more information about the behavior of $K$ near the divisor $X_0$ than the naive
restriction $i^{-1} K$. There is an obvious morphism $i^{-1} K \to \psi_f K$,
and we define the \define{vanishing cycles}
\[
	\phi_f K = \Cone \Bigl( i^{-1} K \to \psi_f K \Bigr).
\]
By construction, there is a canonical morphism $\can \colon \psi_f K \to \phi_f K$;
it is also possible to construct a morphism $\var \colon \phi_f K \to \psi_f K(-1)$
going in the opposite direction.%
\footnote{For historical reasons, Saito denotes this morphism by the symbol
``$\operatorname{Var}$''.}
Both $\psi_f K$ and $\phi_f K$ are constructible
complexes of $\CC$-vector spaces on $X_0$. Gabber showed that when $K$ is a perverse
sheaf, the shifted complexes
\[
	\ppsi_f K = \psi_f K \decal{-1}
		\quad \text{and} \quad 
	\pphi_f K = \phi_f K \decal{-1}
\]
are again perverse sheaves (see \cite[p.~14]{Brylinski} for more information). By
construction, the complex $\ppsi_f K$ has a monodromy operator $T$, induced by the
automorphism $z \mapsto z+1$ of the upper half-plane $\HH$. Since perverse sheaves
form an abelian category, we can decompose into generalized eigenspaces
\[
	\ppsi_{f, \lambda} K = \ker (T - \lambda \id)^m, \qquad \text{$m \gg 0$}.
\]
In summary, we have a decomposition
\[
	\ppsi_f K = \bigoplus_{\lambda \in \CCst} \ppsi_{f, \lambda} K
\]
in the category of perverse sheaves on $X_0$. There is a similar
decomposition for $\pphi_f K$; by construction, $\ppsi_{f, \lambda} K
\simeq \pphi_{f, \lambda} K$ for every complex number $\lambda \neq 1$. On the
unipotent part $\ppsi_{f,1} K$, the composition $\var \circ \can$ is equal to the
nilpotent operator $N = (2 \pi i)^{-1} \log T$; the same goes for $\can \circ \var$
on $\pphi_{f,1} K$.

As the Riemann-Hilbert correspondence would suggest, there is also a notion of nearby
and vanishing cycles for regular holonomic $\Dmod$-modules. It is due to Malgrange
and Kashiwara, and involves the use of an additional filtration
called the \define{Kashiwara-Malgrange filtration}. Such a filtration only exists
when the divisor $X_0$ is smooth; in general, one uses the graph embedding to reduce
to this situation.

Let us first consider the case where we have a smooth function $t \colon X \to \CC$
and a global vector field $\partial_t$ such that $\lbrack \partial_t, t \rbrack = 1$.
Then the \define{Kashiwara-Malgrange filtration} on a right $\Dmod$-module $\Mmod$ is
an increasing filtration $\VKM_{\bullet} \Mmod$, indexed by $\ZZ$, such that
\begin{enumerate}
\item each $\VKM_k \Mmod$ is coherent over $V_0 \Dmod_X = \menge{P \in \Dmod_X}{P \cdot
I_{X_0} \subseteq I_{X_0}}$,
\item $\VKM_k \Mmod \cdot t \subseteq \VKM_{k-1} \Mmod$ and $\VKM_k \Mmod \cdot \partial_t
\subseteq \VKM_{k+1} \Mmod$,
\item $\VKM_k \Mmod \cdot t = \VKM_{k-1} \Mmod$ for $k \ll 0$,
\item all eigenvalues of $t \partial_t$ on $\gr_k^{\VKM} \Mmod = \VKM_k \Mmod /
\VKM_{k-1} \Mmod$ have real part in $(k-1, k \rbrack$.
\end{enumerate}

Kashiwara proved that a filtration with the above properties exists and is unique
provided that $\Mmod$ is holonomic \cite{Kashiwara-V}. One can also show that when
$\Mmod$ is (regular) holonomic, all the graded quotients $\gr_k^{\VKM} \Mmod$ are
again (regular) holonomic $\Dmod$-modules on the complex submanifold $t^{-1}(0) \subseteq X$.

\begin{exercise} 
Prove the following results about the Kashiwara-Malgrange filtration:
\begin{aenumerate}
\item $t \colon \VKM_k \Mmod \to \VKM_{k-1} \Mmod$ is an isomorphism for $k < 0$.
\item $\partial_t \colon \gr_k^{\VKM} \Mmod \to \gr_{k+1}^{\VKM} \Mmod$ is an
isomorphism for $k \neq -1$.
\item $\Mmod$ is generated, as a $\Dmod_X$-module, by the subsheaf  $\VKM_0 \Mmod$.
\end{aenumerate}
\end{exercise}

\begin{exercise}
Calculate the Kashiwara-Malgrange filtration when $\Supp \Mmod \subseteq t^{-1}(0)$.
\end{exercise}

\begin{exercise}
Prove the uniqueness of the Kashiwara-Malgrange filtration by showing that there can
be at most one filtration $\VKM_{\bullet} \Mmod$ with the above properties.
\end{exercise}

\begin{exercise}
Suppose that $\Mmod$ is the right $\Dmod$-module defined by a vector bundle with
integrable connection. Calculate the Kashiwara-Malgrange filtration, and
show that $\gr_{-1}^{\VKM} \Mmod$ is isomorphic to the restriction of $\Mmod$ to the
submanifold $t^{-1}(0)$.
\end{exercise}

Now suppose that $f \colon X \to \CC$ is an arbitrary non-constant holomorphic
function. To reduce to the smooth case, we consider the graph embedding 
\[
	(\id, f) \colon X \into X \times \CC, \qquad x \mapsto \bigl( x, f(x) \bigr).
\]
Let $t$ be the coordinate on $\CC$, and set $\partial_t = \partial/\partial t$; note
that $t \circ (\id, f) = f$. Instead of the original $\Dmod$-module $\Mmod$, we consider
the direct image
\[
	\Mmodf = (\id, f)_{+} \Mmod = \Mmod \lbrack \partial_t \rbrack
\]
on $X \times \CC$; here the action by $\partial_t$ is the obvious one, and the action
by $t$ is defined by using the relation $\lie{\partial_t}{t} = 1$. Since $\Mmod$ and
$\Mmodf$ uniquely determine each other, it makes sense to consider the
Kashiwara-Malgrange filtration on $\Mmodf$ with respect to the smooth function $t$. 
When $\Mmod$ is regular holonomic, the Kashiwara-Malgrange filtration for $\Mmodf$
exists, and the graded quotients $\gr_k^{\VKM} \Mmodf$ are again regular holonomic
$\Dmod$-modules on $X = t^{-1}(0)$ whose support is contained in the original divisor
$X_0 = f^{-1}(0)$. 

The regular holonomic $\Dmod$-modules $\gr_k^{\VKM} \Mmodf$ are related to the
nearby and vanishing cycles for the perverse sheaf $\DR(\Mmod)$ in the following
manner. For any complex number $\alpha \in \CC$, let us define
\[
	\Mmodfal{\alpha} = \ker (t \partial_t - \alpha \id)^m, \qquad \text{$m \gg 0$,}
\]
as the generalized eigenspace with eigenvalue $\alpha$ for the action of $t
\partial_t$ on $\gr_k^{\VKM} \Mmodf$, where $k$ is the unique integer with $k-1 < \Re
\alpha \leq k$.  Then for $\lambda = e^{2\pi i \alpha}$, one has
\begin{equation} \label{eq:Mmodfal}
\begin{split}
	\DR \bigl( \Mmodfal{\alpha} \bigr) & \simeq
		\ppsi_{f, \lambda} \bigl( \DR(\Mmod) \bigr)
		\quad \text{for $-1 \leq \Re \alpha < 0$}, \\
	\DR \bigl( \Mmodfal{\alpha} \bigr) & \simeq
		\pphi_{f, \lambda} \bigl( \DR(\Mmod) \bigr)
		\quad \text{for $-1 < \Re \alpha \leq 0$},
\end{split}
\end{equation}
and the operator $T$ coming from the monodromy action corresponds to 
\[
	e^{2\pi i t \partial_t} 
	= e^{2 \pi i \alpha} \cdot e^{2 \pi i(t \partial_t - \alpha)}
	= \lambda \cdot e^{2 \pi i(t \partial_t - \alpha)}
\]
under this isomorphism. In the same way, the two morphisms
\[
	\can \colon \ppsi_{f,1} K \to \pphi_{f,1} K \quad \text{and} \quad
		\var \colon \pphi_{f,1} K \to \pphi_{f,1} K(-1)
\]
correspond, on the level of $\Dmod$-modules, to
\[
	\partial_t \colon \Mmodfal{-1} \to \Mmodfal{0} \quad \text{and} \quad 
		t \colon \Mmodfal{0} \to \Mmodfal{-1};
\]
the composition $N = \var \circ \can$ is therefore represented by $\partial_t t = t
\partial_t + 1$ (since we are using right $\Dmod$-modules).

\begin{exercise}
Let $(\shV, \nabla)$ be a flat bundle on the punctured disk $\dst$, and let $\Mmod$
be the right $\Dmod$-module corresponding to the Deligne's canonical meromorphic
extension. Show that the Kashiwara-Malgrange filtration with respect to the
coordinate $t$ on $\Delta$ contains exactly the same information as Deligne's
canonical lattices for $(\shV, \nabla)$.
\end{exercise}

\subsection{Nearby and vanishing cycles in Saito's theory} 

For reasons that will become clear in a moment, Saito does not directly use the
Kashiwara-Malgrange filtration; instead, he uses a refined notion that works better
for filtered $\Dmod$-modules. Let $M = (\Mmod, F_{\bullet} \Mmod, K)$ be a 
filtered regular holonomic $\Dmod$-module with $\QQ$-structure on a complex manifold
$X$. Given a non-constant holomorphic function $f \colon X \to \CC$, we would like to
define the nearby cycles $\psi_f M$ and the vanishing cycles $\phi_f M$ in the
category of filtered regular holonomic $\Dmod$-modules with $\QQ$-structure. 

As before, we use the graph embedding $(\id, f) \colon X \into X \times \CC$ and
replace the original filtered $\Dmod$-module $(\Mmod, F_{\bullet} \Mmod)$ by its
direct image 
\begin{equation} \label{eq:def-Mf}
\begin{split}
	\Mmodf &= (\id, f)_{+} \Mmod = \Mmod \lbrack \partial_t \rbrack, \\
	F_{\bullet} \Mmodf &= F_{\bullet} (\id, f)_{+} \Mmod
		= \bigoplus_{i = 0}^{\infty} F_{\bullet-i} \Mmod \tensor \partial_t^i.
\end{split}
\end{equation}
Let $\VKM_{\bullet} \Mmodf$ denote the Kashiwara-Malgrange filtration with respect to
$t = 0$.

Following Saito, we shall only consider objects with quasi-unipotent local monodromy;
this assumption comes from the theory of variations of Hodge structure
\cite{Schmid}. In our situation, it means that all eigenvalues of the monodromy
operator $T$ on $\ppsi_f K$ are roots of unity.%
\footnote{Saito makes this assumption in \cite{Saito-HM,Saito-MHM}, but not in
certain other papers \cite{Saito-Kaehler}.}
Then all eigenvalues of $t \partial_t$ on $\gr_0^{\VKM} \Mmodf$ are rational
numbers, and one can introduce a refined filtration $V_{\bullet} \Mmodf$ indexed by
$\QQ$. Given $\alpha \in \QQ$, we let $k = \lceil \alpha \rceil$, and define
$V_{\alpha} \Mmodf \subseteq \VKM_k \Mmodf$ as the preimage of 
\[
	\bigoplus_{k-1< \beta \leq \alpha} \Mmodfal{\beta}
		\subseteq \gr_k^{\VKM} \Mmodf
\]
under the projection $\VKM_k \Mmodf \to \gr_k^{\VKM} \Mmodf$; we define $V_{<\alpha}
\Mmodf$ in the same way, but taking the direct sum over $k-1 < \beta < \alpha$.
The resulting filtration $V_{\bullet} \Mmodf$ is called the \define{(rational)
$V\!$-filtration}, to distinguish it from the Kashiwara-Malgrange filtration.

\begin{exercise} \label{ex:V-refined}
Prove the following assertions about the $V\!$-filtration:
\begin{aenumerate}
\item The operator $t \partial_t - \alpha \id$ acts nilpotently on $\gr_{\alpha}^V
\Mmod = V_{\alpha} \Mmod \big/ V_{<\alpha} \Mmod$.
\item $t \colon V_{\alpha} \Mmod \to V_{\alpha-1} \Mmod$ is an isomorphism
for $\alpha < 0$.
\item $\partial_t \colon \gr_{\alpha}^V \Mmod \to \gr_{\alpha+1}^V \Mmod$ is an
isomorphism for $\alpha \neq -1$.
\item $V_{<0} \Mmod$ only depends on the restriction of $\Mmod$ to $X \setminus X_0$.
\end{aenumerate}
\end{exercise}

\begin{note}
More generally, one can consider the case where the eigenvalues of $T$ on $\ppsi_f
K$ have absolute value $1$; the $V\!$-filtration is then naturally indexed by $\RR$.
\end{note}

The correct choice of filtration on the nearby and vanishing cycles is a little bit
subtle. Let me explain what the issue is. The formulas in \eqref{eq:Mmodfal},
together with the obvious isomorphisms $\gr_{\alpha}^V \Mmodf \simeq
\Mmodfal{\alpha}$, yield 
\begin{align*}
	\DR \left( \bigoplus_{-1 \leq \alpha < 0} 
		\gr_{\alpha}^V \Mmodf \right) & \simeq \CC \tensor_{\QQ} \ppsi_f K, \\
	\DR \left( \bigoplus_{-1 < \alpha \leq 0} 
		\gr_{\alpha}^V \Mmodf \right) & \simeq \CC \tensor_{\QQ} \pphi_f K,
\end{align*}
and the two $\Dmod$-modules on the left-hand side are regular holonomic. To make them
into filtered $\Dmod$-modules, we should use a filtration that is compatible with the
decompositions above. There are two reasons for this: (1) In the case of a polarizable
variation of Hodge structure on a punctured disk, the correct choice of filtration on
the nearby cycles is the limit Hodge filtration; according to Schmid's nilpotent
orbit theorem \cite{Schmid}, this filtration is compatible with the decomposition
into generalized eigenspaces. (2) On a more formal level, it turns out that certain
properties of $M$ can be described very naturally in terms of the filtration on the
individual $\Dmod$-modules $\gr_{\alpha}^V \Mmodf$; we shall revisit this point below.

The upshot is that we should endow each $\Dmod$-module $\gr_{\alpha}^V \Mmodf$ with
the filtration induced by $F_{\bullet} \Mmodf$; concretely, for $p \in \ZZ$, we set
\[
	F_p \gr_{\alpha}^V \Mmodf 
		= \frac{F_p \Mmodf \cap V_{\alpha} \Mmodf}{F_p \Mmodf \cap V_{<\alpha} \Mmodf}.
\]
Note that this leads to a different filtration on the $\Dmod$-module
\[
	\gr_0^{\VKM} \Mmodf \simeq \bigoplus_{-1 < \alpha \leq 0} \gr_{\alpha}^V \Mmod 
\]
than if we simply took the filtration induced by $F_{\bullet} \Mmodf$; while this
choice might seem more natural, it would be wrong from the point of
view of Hodge theory.

For any non-constant holomorphic function $f \colon X \to \CC$, we can now set
\begin{equation} 
\begin{split}
	\psi_f M &= \bigoplus_{-1 \leq \alpha < 0}
		\Bigl( \gr_{\alpha}^V \Mmodf, F_{\bullet-1} \gr_{\alpha}^V \Mmodf, 
			\ppsi_{f, e^{2 \pi i \alpha}} K \Bigr) \\
	\psi_{f,1} M &= 
		\Bigl( \gr_{-1}^V \Mmodf, F_{\bullet-1} \gr_{-1}^V \Mmodf, 
			\ppsi_{f, 1} K \Bigr) \\
	\phi_{f,1} M &= 
		\Bigl( \gr_0^V \Mmodf, F_{\bullet} \gr_0^V \Mmodf, \pphi_{f,1} K \Bigr)
\end{split}
\end{equation}
Except for $\lambda = 1$, the individual perverse sheaves $\ppsi_{f, \lambda} K$ are
generally not defined over $\QQ$; in order to have a $\QQ$-structure on
$\psi_f M$, we are forced to keep them together.  Provided that the induced
filtration $F_{\bullet} \gr_{\alpha}^V \Mmodf$ is good for every $\alpha \in [-1,
0]$, all three objects are filtered regular holonomic $\Dmod$-modules with
$\QQ$-structure on $X$; by construction, their support is contained in the divisor
$f^{-1}(0)$. The logarithm of the unipotent part of the monodromy 
\[
	N = \frac{1}{2\pi i} \log T_u \colon \psi_f M \to \psi_f M(-1)
\]
is a nilpotent endomorphism (up to a Tate twist).

\subsection{Decomposition by strict support}

Now we can start thinking about the definition of Hodge modules. Recall that the
category $\HMp{X}{w}$ should be such that every polarizable Hodge module decomposes
into a finite sum
\[
	M \simeq \bigoplus_{Z \subseteq X} M_Z,
\]
where $M_Z$ is supported on an irreducible subvariety $Z$, and in fact comes from a
variation of Hodge structure of weight $w - \dim Z$ on a Zariski-open subset of the
smooth locus of $Z$. On the level of perverse sheaves, this means that $K_Z$ should
be an intersection complex, and should therefore not have nontrivial subobjects or
quotient objects that are supported on proper subvarieties of $Z$. The following
definition captures this property in the case of $\Dmod$-modules.

\begin{definition}
Let $Z \subseteq X$ be an irreducible subvariety. We say that a $\Dmod$-module
$\Mmod$ has \define{strict support} $Z$ if the support of every nonzero subobject
or quotient object of $\Mmod$ is equal to $Z$.
\end{definition}

Note that a regular holonomic $\Dmod$-module has strict support $Z$ if and only if
the corresponding perverse sheaf $\DR(\Mmod)$ is the intersection complex of a local
system on a dense Zariski-open subset of $Z$. It turns out that this property can
also be detected with the help of the $V\!$-filtration.  

\begin{exercise} \label{ex:SS}
Let $f \colon X \to \CC$ be a non-constant holomorphic function, and let $\Mmod$ be a
regular holonomic $\Dmod$-module on $X$.
\begin{aenumerate}
\item Show that $\Mmod$ has no nonzero subobject supported on $f^{-1}(0)$ if and only
if $t \colon \gr_0^V \Mmodf \to \gr_{-1}^V \Mmodf$ is injective.
\item Show that $\Mmod$ has no nonzero quotient object supported on $f^{-1}(0)$ if
and only if $\partial_t \colon \gr_{-1}^V \Mmodf \to \gr_0^V \Mmodf$ is surjective.
\end{aenumerate}
\end{exercise}

More generally, one can use the $V\!$-filtration to test whether or not $\Mmod$
decomposes into a direct sum of $\Dmod$-modules with strict support.

\begin{exercise}
Let $f \colon X \to \CC$ be a non-constant holomorphic function, and let $\Mmod$ be a
regular holonomic $\Dmod$-module on $X$. Show that
\begin{equation} \label{eq:SS-test}
	\gr_0^V \Mmodf = \ker \bigl( t \colon \gr_0^V \Mmodf \to \gr_{-1}^V \Mmodf \bigr)
		\oplus \im \bigl( \partial_t \colon \gr_{-1}^V \Mmodf \to \gr_0^V \Mmodf \bigr)
\end{equation}
if and only if $\Mmod \simeq \Mmod' \oplus \Mmod''$, where $\Supp \Mmod' \subseteq
f^{-1}(0)$, and $\Mmod''$ does not have nonzero subobjects or quotient objects whose
support is contained in $f^{-1}(0)$.
\end{exercise}

Using the result of the preceding exercise, one can easily prove the following local
criterion for the existence of a decomposition by strict support. Note that such a
decomposition is necessarily unique, because there are no nontrivial morphisms
between $\Dmod$-modules with different strict support. 

\begin{proposition} \label{prop:SS-criterion}
Let $\Mmod$ be a regular holonomic $\Dmod$-module on $X$. Then $\Mmod$ decomposes
into a direct sum of $\Dmod$-modules with strict support if and only if
\eqref{eq:SS-test} is true for every $f$.
\end{proposition}

This is the first indication that nearby and vanishing cycle functors are useful in
describing global properties of $\Dmod$-modules in terms of local conditions.

\subsection{Compatibility with the filtration}

Let $(\Mmod, F_{\bullet} \Mmod, K)$ be a filtered regular holonomic $\Dmod$-module
with $\QQ$-structure on a complex manifold $X$. Saito realized that the nearby and
vanishing cycle functors are also good for imposing restrictions on the filtration
$F_{\bullet} \Mmod$. To see why this might be the case, let us suppose for a moment
that $\Mmod$ has strict support $Z \subseteq X$. Let $f \colon X \to \CC$ be a
holomorphic function whose restriction to $Z$ is not constant; the most interesting
case is when $X_0$ contains the singular locus of $\Mmod$, or in other words,
when $K$ is the intersection complex of a local system on $Z \setminus Z \cap X_0$. 

To simplify the notation, suppose that we are actually dealing with a smooth function
$t \colon X \to \CC$, and that there is a global vector field $\partial_t$ with
$\lbrack \partial_t, t \rbrack = 1$. Clearly, we can always put ourselves into this
situation by considering the $\Dmod$-module $\Mmodf$ on the product $X \times \CC$:
it will have strict support $i_f(Z)$, and the two filtrations $F_{\bullet} \Mmod$ and
$F_{\bullet} \Mmodf$ determine each other by \eqref{eq:def-Mf}. 
We know from \exerciseref{ex:SS} that
\begin{align*}
	t &\colon \gr_0^V \Mmod \to \gr_{-1}^V \Mmod \quad \text{is injective,} \\
	\partial_t &\colon \gr_{-1}^V \Mmod \to \gr_0^V \Mmod \quad \text{is surjective.}
\end{align*}
Because of the properties of the $V\!$-filtration, one then has
\[
	\Mmod = \sum_{i = 0}^{\infty} \bigl( V_{<0} \Mmod \bigr) \partial_t^i,
\]
and so $V_{<0} \Mmod$ generates $\Mmod$ as a right $\Dmod$-module; recall from
\exerciseref{ex:V-refined} that this
is the part of the $V\!$-filtration that only depends on the restriction of $\Mmod$ to
$X \setminus X_0$. We would like to make sure that the filtration $F_{\bullet} \Mmod$
is also determined by its restriction to $X \setminus X_0$. First, we need a
condition that allows us to recover the coherent sheaves
\[
	F_p V_{<0} \Mmod = V_{<0} \Mmod \cap F_p \Mmod.
\]
from information on $X \setminus X_0$.

\begin{exercise}
Denote by $j \colon X \setminus X_0 \into X$ the inclusion. Show that one has
\[
	F_p V_{<0} \Mmod = V_{<0} \Mmod \cap \jl \ju F_p \Mmod
\]
if and only if $t \colon F_p V_{\alpha} \Mmod \to F_p V_{\alpha - 1} \Mmod$ is
surjective for every $\alpha < 0$.
\end{exercise}

Next, we observe that $F_p \Mmod$ contains $\bigl( F_{p-i} V_{<0} \Mmod \bigr)
\partial_t^i$ for every $i \geq 0$. The following exercise gives a criterion for when
these subsheaves generate $F_p \Mmod$.

\begin{exercise}
Suppose $\partial_t \colon \gr_{-1}^V \Mmod \to \gr_0^V \Mmod$ is surjective. Show
that one has
\[
	F_p \Mmod = \sum_{i = 0}^{\infty} \bigl( F_{p-i} V_{<0} \Mmod \bigr) \partial_t^i
\] 
if and only if $\partial_t \colon F_p \gr_{\alpha}^V \Mmod \to F_{p+1} \gr_{\alpha+1}^V
\Mmod$ is surjective for every $\alpha \geq -1$.  
\end{exercise}

This criterion is one reason for considering the induced filtration on each
$\gr_{\alpha}^V \Mmod$. Almost exactly the same condition also turns out to be useful
for describing the filtration in the case where the support of $\Mmod$ is contained
in the divisor $X_0$.

\begin{exercise}
Suppose that $\Supp \Mmod \subseteq X_0$, and define a filtered $\Dmod$-module on
$X_0$ by setting $\Mmod_0 = \ker(t \colon \Mmod \to \Mmod)$ and $F_p \Mmod_0 = F_p
\Mmod \cap \Mmod_0$. Show that 
\[
	\Mmod \simeq \Mmod_0 \lbrack \partial_t \rbrack \quad \text{and} \quad
		F_p \Mmod \simeq \bigoplus_{i = 0}^{\infty} F_{p-i} \Mmod_0 \tensor \partial_t^i
\]
if and only if $\partial_t \colon F_p \gr_{\alpha}^V \Mmod \to F_{p+1} \gr_{\alpha+1}^V
\Mmod$ is surjective for every $\alpha > -1$.  
\end{exercise}

We now return to the case of an arbitrary filtered regular holonomic $\Dmod$-module
$(\Mmod, F_{\bullet} \Mmod)$. Let $f \colon X \to \CC$ be a non-constant holomorphic
function. The three exercises above motivate the following definition.

\begin{definition} \label{def:strict-spec}
We say that $(\Mmod, F_{\bullet} \Mmod)$ is \define{quasi-unipotent along $f = 0$} if
all eigenvalues of the monodromy operator on $\ppsi_f K$ are roots of unity, and if 
the $V\!$-filtration $V_{\bullet} \Mmodf$ satisfies the following two
additional conditions:
\begin{nenumerate}
\item 
$t \colon F_p V_{\alpha} \Mmodf \to F_p V_{\alpha-1} \Mmodf$ is surjective for
$\alpha < 0$.
\item
$\partial_t \colon F_p \gr_{\alpha}^V \Mmodf \to F_{p+1} \gr_{\alpha+1}^V \Mmodf$
is surjective for $\alpha > -1$.
\end{nenumerate}
We say that $(\Mmod, F_{\bullet} \Mmod)$ is \define{regular along $f=0$} if
$F_{\bullet} \gr_{\alpha}^V \Mmodf$ is a good filtration for every $-1 \leq \alpha \leq 0$.
\end{definition}

The properties of the $V\!$-filtration guarantee that the two morphisms in the
definition are isomorphisms in the given range. Note that we do not include $\alpha =
-1$ on the second line, because we want a notion that also makes sense when
$\partial_t \colon \gr_{-1}^V \Mmodf \to \gr_0^V \Mmodf$ is not surjective. The
regularity condition ensures that $\psi_f M$ and $\phi_{f,1} M$ are filtered regular
holonomic $\Dmod$-modules with $\QQ$-structure.

To connect the definition with the discussion above, suppose that $(\Mmod,
F_{\bullet} \Mmod)$ has strict support $Z$, and that it is quasi-unipotent and
regular along $f = 0$ for a holomorphic function $f \colon X \to \CC$ whose
restriction to $Z$ is not constant. Then
\begin{equation} \label{eq:F-formula}
	F_p \Mmodf = \sum_{i = 0}^{\infty} 
		\bigl( V_{<0} \Mmodf \cap \jl \ju F_{p-i} \Mmodf \bigr) \partial_t^i,
\end{equation}
provided that $\partial_t \colon F_p \gr_{-1}^V \Mmodf \to F_{p+1} \gr_0^V \Mmodf$ is
also surjective. This last condition will be automatically satisfied for Hodge
modules: the recursive definition (in \parref{par:def-HM}) implies that the morphism
\[
	\partial_t \colon \gr_{-1}^V \Mmodf \to \gr_0^V \Mmodf
\]
is strictly compatible with the filtrations (up to a shift by $1$), and
\eqref{eq:F-formula} therefore follows from the surjectivity of $\partial_t$. What
this means is that the filtered $\Dmod$-module $(\Mmod, F_{\bullet} \Mmod)$ is
uniquely determined by its restriction to $Z \setminus Z \cap X_0$.

Another good feature of the conditions in \definitionref{def:strict-spec} is that
they say something interesting even when the support of $\Mmod$ is contained in the
divisor $f^{-1}(0)$.

\begin{exercise} \label{ex:support}
Suppose that $\Supp \Mmod \subseteq f^{-1}(0)$. Show that $(\Mmod, F_{\bullet}
\Mmod)$ is quasi-unipotent and regular along $f = 0$ if and only if
the filtration satisfies 
\[
	(F_p \Mmod) \cdot f \subseteq F_{p-1} \Mmod
\]
for every $p \in \ZZ$.
\end{exercise}

We can also upgrade the criterion in \propositionref{prop:SS-criterion} to filtered
regular holonomic $\Dmod$-modules with $\QQ$-structure. In the statement, note that
the filtration on the image of $\can \colon \psi_{f,1} M \to \phi_{f,1} M$ is induced
by that on $\psi_{f,1} M$.  

\begin{theorem} \label{thm:SS-criterion}
Let $M$ be a filtered regular holonomic $\Dmod$-module with $\QQ$-structure, and
suppose that $(\Mmod, F_{\bullet} \Mmod)$ is quasi-unipotent and regular along
$f = 0$ for every locally defined holomorphic function $f$. Then $M$ admits a
decomposition 
\[
	M \simeq \bigoplus_{Z \subseteq X} M_Z
\]
by strict support, in which each $M_Z$ is again a filtered regular holonomic
$\Dmod$-module with $\QQ$-structure, if and only if one has
\[
	\phi_{f,1} M = \ker \bigl( \var \colon \phi_{f,1} M \to \psi_{f,1} M(-1) \bigr)
		\oplus \im \bigl( \can \colon \psi_{f,1} M \to \phi_{f,1} M \bigr)
\]
for every locally defined holomorphic function $f$.
\end{theorem}

\subsection{Definition of pure Hodge modules}
\label{par:def-HM}

After these technical preliminaries, we are now ready to define the category of pure
Hodge modules. The basic objects are filtered regular holonomic $\Dmod$-modules with
$\QQ$-structure; we have to decide when a given $M = (\Mmod, F_{\bullet} \Mmod, K)$
should be called a Hodge module. Saito uses a recursive procedure to define a family
of auxiliary categories 
\[
	\HMs{d}{X}{w} = \left\{ \, \text{\parbox{5.2cm}{Hodge modules of weight $w$ on $X$ \\
			whose support has dimension $\leq d$}} \, \right\},
\]
indexed by $d \geq 0$. This procedure has the advantage that results can then be
proved by induction on the dimension of the support.%
\footnote{Saito says that he found the definition by axiomatizing certain
arguments that he used to prove \theoremref{thm:direct-image} in the case $M = \QQ_X^H
\decal{\dim X}$.}
Since the definition involves the nearby and vanishing cycle functors, we impose the
following condition: 
\begin{equation} \label{eq:HM-def-1}
	\text{\parbox{10cm}{%
		The pair $(\Mmod, F_{\bullet} \Mmod)$ is quasi-unipotent and regular along $f=0$ \\ 
		for every locally defined holomorphic function $f \colon U \to \CC$.}}
\end{equation}
Both $\psi_f M$ and $\phi_{f,1} M$ are then well-defined filtered regular holonomic
$\Dmod$-modules with $\QQ$-structure on $U$, whose support is contained in
$f^{-1}(0)$. Since we are only interested in objects that admit a
decomposition by strict support, we require:
\begin{equation} \label{eq:HM-def-2}
	\text{\parbox{10cm}{$M$ admits a decomposition by strict support,
		in the category \\ of regular holonomic $\Dmod$-modules with $\QQ$-structure.}}
\end{equation}
Recall that this can also be tested with the help of the nearby and vanishing cycle
functors, using the criterion in \theoremref{thm:SS-criterion}.  Now the problem is
reduced to defining, for each irreducible closed subvariety $Z \subseteq X$, a
suitable category
\[
	\HMZ{Z}{X}{w} = \left\{ \, \text{\parbox{5cm}{Hodge modules of weight $w$ on $X$ \\
			with strict support equal to $Z$}} \, \right\};
\]
we can then take for $\HMs{d}{X}{w}$ the direct sum of all $\HMZ{Z}{X}{w}$ with $\dim
Z \leq d$.

The first case to consider is obviously when $Z = \{x\}$ is a point. Here the
perverse sheaf $K$ is the direct image of a $\QQ$-vector space by the morphism
$i \colon \{x\} \into X$, and this suggests defining
\begin{equation} \label{eq:HM-def-3}
	\HMZ{\{x\}}{X}{w} = \menge{\il H}{%
		\text{$H$ is a $\QQ$-Hodge structure of weight $w$}}.
\end{equation}
Now comes the most important part of the definition: for
arbitrary $Z \subseteq X$, we say that $M \in \HMZ{Z}{X}{w}$ if and only if, for
every locally defined holomorphic function $f \colon U \to \CC$ that does not vanish
identically on $Z \cap U$, one has
\begin{equation} \label{eq:HM-def-4}
	\gr_{\ell}^{W(N)}(\psi_f M) \in \HMs{d-1}{X}{w-1+\ell}.
\end{equation}
Here $W(N)$ is the monodromy filtration of the nilpotent operator $N = (2\pi i)^{-1}
\log T_u$ on the nearby cycles $\psi_f M$. One can deduce from this last condition
that 
\[
	\gr_{\ell}^{W(N)}(\phi_{f,1} M) \in \HMs{d-1}{X}{w+\ell},
\]
by using the isomorphism in \theoremref{thm:SS-criterion}. Of course, the motivation
for using the monodromy filtration comes from Schmid's $\SL(2)$-orbit theorem
\cite{Schmid}: in the case of a polarizable variation of Hodge structure on the
punctured disk, the nearby cycles carry a mixed Hodge structure whose Hodge
filtration is the limit Hodge filtration and whose weight filtration is the monodromy
filtration of $N$, up to a shift by the weight of the variation.

\begin{definition}
The category of \define{Hodge modules of weight $w$} on $X$ has objects
\[
	\HM{X}{w} = \bigcup_{d \geq 0} \HMs{d}{X}{w} 
		= \bigoplus_{Z \subseteq X} \HMZ{Z}{X}{w};
\]
its morphisms are the morphisms of regular holonomic $\Dmod$-modules with
$\QQ$-structure.
\end{definition}

One can show that every morphism between two Hodge modules is strictly compatible
with the filtrations, and that $\HM{X}{w}$ is therefore an abelian category.
Moreover, one can show that there are no nontrivial morphisms from one Hodge module
to another Hodge module of strictly smaller weight.

An important point is that all the conditions in the definition are local; Saito's
insight is that they are strong enough to have global consequences, such as the
decomposition theorem. Of course, the recursive nature of the definition makes it
hard to prove that a given $M$ is a Hodge module: in fact, it takes considerable work
to establish that even a very basic object like
\[
	\bigl( \omX, F_{\bullet} \omX, \QQ_X \decal{n} \bigr)
\]
is a Hodge module of weight $n = \dim X$. It is also not clear that Hodge modules are
stable under various operations such as direct or inverse images -- about the only
thing that is obvious from the definition is that 
\[
	 \HM{pt}{w} = \bigl\{ \text{$\QQ$-Hodge structures of weight $w$} \bigr\}. 
\]
On the positive side, Hodge modules are by definition stable under the application of
nearby and vanishing cycle functors: once we know that something is a Hodge module,
we immediately get a large collection of other Hodge modules by taking nearby and
vanishing cycles.

\begin{example}
Suppose $M$ is a Hodge module on a smooth curve $X$. Then for every local coordinate
$t$, the nearby cycles $\psi_t M$ carry a mixed Hodge structure.
\end{example}

\subsection{Polarizations}

In order to define polarizable Hodge modules, we also need to introduce the concept
of a polarization. Let $M = (\Mmod, F_{\bullet} \Mmod, K)$ be a filtered regular
holonomic $\Dmod$-module with $\QQ$-structure on an $n$-dimensional complex manifold
$X$; of course, we will be mostly interested in the case $M \in \HM{X}{w}$.

Recall that in the case of a variation of Hodge structure $V$ of weight $w-n$, a
polarization is a bilinear form $V \tensor_{\QQ} V \to \QQ(-w+n)$
whose restriction to every point $x \in X$ polarizes the Hodge structure
$V_x$; equivalently, it is a $\QQ$-linear mapping $V(w-n) \to V^{\ast}$ with the same
property. The natural analogue for perverse sheaves is to consider
morphisms of the form
\[
	K(w) \to \DD K,
\]
where $\DD K = \derR \shHom \bigl( K, \QQ_X(n) \bigr) \decal{2n}$ is the Verdier dual
of the perverse sheaf $K$, with a Tate twist to conform to Saito's
conventions for weights. This suggests that a polarization should be a
morphism $M(w) \to \DD M$ \dots\ except that there is no duality functor for
arbitrary filtered $\Dmod$-modules. The holonomic dual
\[
	R^n \shHom_{\Dmod_X} \bigl( \Mmod, \omX \tensor_{\OX} \Dmod_X \bigr)
\]
of $\Mmod$ is a regular holonomic $\Dmod$-module, whose de Rham complex is isomorphic
to $\DD \bigl( \DR(\Mmod) \bigr)$; but this operation does not interact well with the
filtration $F_{\bullet} \Mmod$. In fact, a necessary and sufficient condition is that
$(\Mmod, F_{\bullet} \Mmod)$ is Cohen-Macaulay, meaning that $\gr_{\bullet}^F \!
\Mmod$ should be a Cohen-Macaulay module over $\gr_{\bullet}^F \!  \Dmod_X$.
Fortunately, Saito has shown that Hodge modules always have this property, and that
every $M \in \HM{X}{w}$ has a well-defined dual $\DD M \in \HM{X}{-w}$; see
\parref{par:duality}.

With this issue out of the way, here is the definition. A \define{polarization} on a
Hodge module $M \in \HM{X}{w}$ is a morphism $K(w) \to \DD K$ with the following
properties:
\begin{enumerate}
\item It is nondegenerate and compatible with the filtration, meaning that it extends to
an isomorphism $M(w) \simeq \DD M$ in the category of Hodge modules.
\item For every summand $M_Z$ in the decomposition of $M$ by strict support, and for
every locally defined holomorphic function $f \colon U \to \CC$ that is not
identically zero on $U \cap Z$, the induced morphism 
\[
	\ppsi_f K_Z(w) \to \DD(\ppsi_f K_Z)
\]
is a polarization of Hodge-Lefschetz type (= on primitive parts for $N$).
\item If $\dim \Supp M_Z = 0$, then $K_Z(w) \to \DD K_Z$ is induced by
a polarization of Hodge structures in the usual sense.
\end{enumerate}

The second condition uses the compatibility of the duality functor with nearby
cycles, and the fact that $K(w) \to \DD K$ is automatically compatible with the
decomposition by strict support (because there are nonontrivial morphisms
between perverse sheaves with different strict support). Since $\ppsi_f K_Z$ is by
construction supported in a subset of smaller dimension, the
definition is again recursive.

\begin{example}
Let $M$ be the filtered regular holonomic $\Dmod$-module with $\QQ$-structure
associated with a variation of Hodge structure of weight $w$. Provided one is
sufficiently careful with signs, a polarization $V \tensor_{\QQ} V \to \QQ(-w)$
determines a morphism $M(w + \dim X) \to \DD M$; we will see later that it is a
polarization in the above sense.
\end{example}

Note that we do not directly require any ``positivity'' of the morphism giving the
polarization; instead, we are asking that once we apply sufficiently many nearby
cycle functors to end up with a vector space, certain induced Hermitian forms on this
vector space are positive definite. Because of the indirect nature of this definition,
the problem of signs is quite nontrivial, and one has to be extremely careful in
choosing the signs in various isomorphisms. For that reason, Saito actually defines
polarizations not as morphisms $K(w) \to \DD K$, but as pairings $K \tensor_{\QQ} K
\to \QQ(-w) \decal{2n}$; this is equivalent to our definition, but makes it
easier to keep track of signs.

\begin{definition}
We say that a Hodge module is \define{polarizable} if it admits at least one
polarization, and we denote by 
\[
	\HMp{X}{w} \subseteq \HM{X}{w} \quad \text{and} \quad
		\HMZp{Z}{X}{w} \subseteq \HMZ{Z}{X}{w}
\]
the full subcategories of polarizable Hodge modules.
\end{definition}

\subsection{Kashiwara's equivalence and singular spaces}
\label{par:Kashiwara}

To close this chapter, let me say a few words about the definition of Hodge modules
on analytic spaces; a systematic treatment of $\Dmod$-modules on analytic spaces can
be found in \cite{Saito-D}.  The idea is similar to how one defines holomorphic
functions on analytic spaces: embed a given analytic space $X$ into a complex
manifold, and consider objects on the ambient manifold that are supported on
$X$. Of course, such embeddings may only exist locally, and so in the most general
case, one has to cover $X$ by embeddable open subsets and impose conditions on the
pairwise intersections.

To simplify the discussion, we shall only consider those analytic spaces $X$ that are
globally embeddable into a complex manifold; this class includes, for example, all
quasi-projective algebraic varieties. By definition, a right $\Dmod$-module on
$X$ is a right $\Dmod$-module on the ambient manifold whose support is contained in
$X$. This definition is independent of the choice of embedding because of the
following fundamental result by Kashiwara \cite{Kashiwara-th}.

\begin{proposition}[Kashiwara's Equivalence] \label{prop:Kashiwara}
Let $X$ be a complex manifold, and $i \colon Z \into X$ the inclusion of a
closed submanifold. Then the direct image functor
\[
	\ip \colon \Mcoh(\Dmod_Z) \to \Mcoh(\Dmod_X)
\]
gives an equivalence between the category of coherent right $\Dmod_Z$-modules
and the category of coherent right $\Dmod_X$-modules whose support is contained in $Z$.
\end{proposition}

\begin{exercise} \label{ex:Kashiwara-inverse}
Suppose that $Z = t^{-1}(0)$ for a holomorphic function $t \colon X \to \CC$. Prove
Kashiwara's equivalence in this case, by using the decomposition given by the kernels
of the operators $t \partial_t - i$ for $i \geq 0$.
\end{exercise}

Since the direct image functor for left $\Dmod$-modules involves an additional twist
by the relative canonical bundle $\omega_Z \tensor \omX^{-1}$, both Kashiwara's
equivalence and the definition of $\Dmod$-modules on singular spaces are more natural
for right $\Dmod$-modules.

\begin{definition}
Let $X$ be an analytic space that can be globally embedded into a complex manifold
$Y$. Then we define 
\[
	\HM{X}{w} \subseteq \HM{Y}{w} \quad \text{and} \quad
	\HMp{X}{w} \subseteq \HMp{Y}{w}
\]
by taking all (polarizable) Hodge modules on $Y$ whose support is contained in $X$.
\end{definition}

One subtle point is that Kashiwara's equivalence is not true for filtered
$\Dmod$-modules, because a coherent sheaf with support in $X$ is not the same thing
as a coherent sheaf on $X$.  To prove that the category $\HM{X}{w}$ is independent of
the choice of embedding, one has to use the fact that the filtered $\Dmod$-modules in
Saito's theory are always quasi-unipotent and regular.

\begin{exercise} \label{ex:Kashiwara-MF}
In the notation of \propositionref{prop:Kashiwara}, suppose that $\Supp \Mmod
\subseteq Z$, and that $(\Mmod, F_{\bullet} \Mmod)$ is quasi-unipotent and regular
along $f = 0$ for every locally defined holomorphic function in the ideal of $Z$.
Show that 
\[
	(\Mmod, F_{\bullet} \Mmod) \simeq \ip(\Mmod_Z, F_{\bullet} \Mmod_Z)
\]
for a filtered $\Dmod$-module $(\Mmod_Z, F_{\bullet} \Mmod_Z)$ on $Z$, unique up to
isomorphism.
\end{exercise}

Now let $M \in \HM{X}{w}$. The underlying regular holonomic right $\Dmod$-module
$\Mmod$, as well as the coherent sheaves $F_{\bullet} \Mmod$ in the Hodge filtration,
are really defined on the ambient complex manifold $Y$; they are not objects on $X$,
although the support of $\Mmod$ is contained in $X$.  On the other hand, the result in
\exerciseref{ex:support} can be used to show that each $\gr_k^F \Mmod$ is a coherent sheaf
on $X$.

\begin{exercise} \label{ex:DR}
Let $M \in \HM{X}{w}$. Deduce from \exerciseref{ex:Kashiwara-MF} that the graded
quotients $\gr_k^F \DR(\Mmod)$ of the de Rham complex are well-defined complexes of
coherent sheaves on $X$, independent of the choice of embedding.  
\end{exercise}

\section{Two important theorems}

In this chapter, we describe two of Saito's most important results: the structure
theorem and the direct image theorem. The structure theorem relates polarizable Hodge
modules and polarizable variations of Hodge structure; the direct image theorem says
that polarizable Hodge modules are stable under direct images by projective
morphisms.  Taken together, they justify the somewhat complicated definition of the
category $\HMp{X}{w}$.

\subsection{Structure theorem}

One of the main results in Saito's second paper \cite{Saito-MHM} is that polarizable
Hodge modules on $X$ with strict support $Z$ are the same thing as generically
defined polarized variations of Hodge structure on $Z$. 

\begin{theorem}[Structure Theorem] \label{thm:structure}
Let $X$ be a complex manifold, and $Z \subseteq X$ an irreducible closed
analytic subvariety.
\begin{enumerate}
\item Every polarizable variation of $\QQ$-Hodge structure of weight $w - \dim Z$ on a
Zariski-open subset of $Z$ extends uniquely to an object of $\HMZp{Z}{X}{w}$.
\item Every object of $\HMZp{Z}{X}{w}$ is obtained in this way.
\end{enumerate}
\end{theorem}

Together with the condition \eqref{eq:HM-def-2} in the definition of Hodge modules,
this result implies that every polarizable Hodge $M \in \HM{X}{w}$ is of the form
\[
	M = \bigoplus_{Z \subseteq X} M_Z,
\]
where $M_Z$ is obtained from a polarizable variation of Hodge structure of
weight $w - \dim Z$ on a Zariski-open subset of the smooth locus of $Z$; conversely,
every object of this type is a polarizable Hodge module. We have therefore achieved
our goal, which was to describe this class in terms of local conditions.

\begin{example}
On every complex manifold $X$,
\[	
	\QQ_X^H \decal{n} = \bigl( \omX, F_{\bullet} \omX, \QQ_X \decal{n} \bigr) 
		\in \HMZp{X}{X}{n}
\]
is a polarizable Hodge module of weight $n = \dim X$.
\end{example}

\begin{example}
More generally, we can consider the constant variation of Hodge structure on the
smooth locus of an irreducible analytic subvariety $Z \subseteq X$. By
\theoremref{thm:structure}, it determines a polarizable Hodge module of weight $\dim
Z$ on $X$; the underlying perverse sheaf is the intersection complex of $Z$.
\end{example}

\begin{example}
Another consequence of the structure theorem is that the inverse image of a polarizable
Hodge module under a smooth morphism $f \colon Y \to X$ is again a polarizable Hodge
module (with a shift in weight by $\dim Y - \dim X$); see \parref{par:inverse-image}.
This statement looks innocent enough, but trying to prove it directly from the
definition is hopeless, because there are too many additional functions on $Y$.
\end{example}

The proof of the second assertion in \theoremref{thm:structure} is not hard. In
fact, given $M \in \HMZp{Z}{X}{w}$, there is a Zariski-open subset of the smooth locus
of $Z$ where $\rat M$ is a local system (up to a shift), and one has all the data
necessary to define a polarizable variation of Hodge structure of weight $w - \dim
Z$. One then uses the definition to argue that the variation of Hodge structure uniquely
determines the original $M$: for example, \eqref{eq:F-formula} shows how to recover
$F_{\bullet} \Mmod$. The real content is therefore in the first
assertion; we shall say more about the proof in \parref{par:structure-proof} below.

\subsection{Direct image theorem}
\label{par:direct-image-thm}

Given the local nature of the definition, it is also not clear that the category of Hodge
modules is preserved by the direct image functor. The main result in Saito's first
paper \cite{Saito-HM} is that this is true for direct images by projective morphisms.
Along the way, Saito also proved the decomposition theorem for those perverse sheaves
that underlie polarizable Hodge modules. 

Let $f \colon X \to Y$ be a projective morphism between complex manifolds, and let
\[
	M = (\Mmod, F_{\bullet} \Mmod, K)
\]
be a filtered regular holonomic $\Dmod$-module with $\QQ$-structure. In a suitable
derived category, one can define the direct image $\fl M$: the underlying
constructible complex is $\derR \fl K$, and the underlying complex of filtered
$\Dmod$-modules is $\fp(\Mmod, F_{\bullet} \Mmod)$; see \parref{par:direct-image} for
more details. The content of the direct image theorem is, roughly speaking, that $M
\in \HMp{X}{w}$ implies $\shH^i \fl M \in \HMp{Y}{w+i}$. A tricky point is that the
cohomology sheaves $\shH^i \fp(\Mmod, F_{\bullet} \Mmod)$ are in general not filtered
$\Dmod$-modules, but live in a larger abelian category (see \parref{par:DF} below).
Unless the complex $\fp(\Mmod, F_{\bullet} \Mmod)$ is what is called ``strict'', it
is therefore not possible to define $\shH^i \fl M$ as a filtered regular
holonomic $\Dmod$-module with $\QQ$-structure.

Before we can give the precise statement of the direct image theorem, we need to
introduce the Lefschetz operator; it plays an even greater role here than in
classical Hodge theory. Let $\ell \in H^2 \bigl( X, \ZZ(1) \bigr)$ be the first Chern
class of a relatively ample line bundle on $X$. It gives rise to a morphism 
$\ell \colon K \to K(1) \decal{2}$ in the derived category of constructible complexes
on $X$; using the fact that $\ell$ can also be represented by a closed $(1,1)$-form,
one can lift this morphism to a morphism $\ell \colon M \to M(1) \decal{2}$. Now we
apply the direct image functor; provided that $\fp(\Mmod, F_{\bullet} \Mmod)$ is strict, 
we obtain a collection of morphisms
\[
	\ell \colon \shH^i \fl M \to \shH^{i+2} \fl M(1),
\]
which together constitute the \define{Lefschetz operator}.

\begin{theorem}[Direct Image Theorem] \label{thm:direct-image}
Let $f \colon X \to Y$ be a projective morphism between two complex manifolds, and
let $M \in \HMp{X}{w}$. Then one has:
\begin{enumerate}
\item The complex $\fp(\Mmod, F_{\bullet} \Mmod)$ is strict, and $\shH^i \fl M \in
\HMp{Y}{w+i}$.
\item For every $i \geq 0$, the morphism
\[
	\ell^i \colon \shH^{-i} \fl M \to \shH^i \fl M(i)
\]
is an isomorphism of Hodge modules.
\item Any polarization on $M$ induces a polarization on $\bigoplus_i \shH^i \fl M$ in
the Hodge-Lefschetz sense (= on primitive parts for $\ell$).
\end{enumerate}
\end{theorem}

\begin{example}
In the case where $f \colon X \to \pt$ is a morphism from a smooth projective variety
to a point, the direct image theorem says that the $i$-th cohomology group $H^i(X,
K)$ of the perverse sheaf $K$ carries a polarized Hodge structure of
weight $w + i$, and that the hard Lefschetz theorem holds.
\end{example}

The most notable consequence of \theoremref{thm:direct-image} is the decomposition
theorem for perverse sheaves underlying polarizable Hodge modules. This result, which
contains \cite[Th\'eor\`eme~6.2.5]{BBD} as a special case, is one of the major
accomplishments of Saito's theory.%
\footnote{De Cataldo and Migliorini \cite{dCM} later found a more elementary
Hodge-theoretic proof for the decomposition theorem in the special case $K = \QQ_X
\decal{\dim X}$.}

\begin{corollary} \label{cor:decomposition-K}
Let $f \colon X \to Y$ be a projective morphism between complex manifolds, and
$K$ a perverse sheaf that underlies a polarizable Hodge module. Then
\[
	\derR \fl K \simeq \bigoplus_{i \in \ZZ} \bigl( \pshH^i \fl K \bigr) \decal{-i}
\]
in the derived category of perverse sheaves on $Y$.
\end{corollary}

\begin{proof}
We have a  morphism $\ell \colon \derR \fl K \to \derR \fl K(1) \decal{2}$ with the
property that
\[
	\ell^i \colon \pshH^{-i} \fl K \to \pshH^i \fl K(i)
\]
is an isomorphism for every $i \geq 0$. We can now apply a result by Deligne
\cite{Deligne-Lef} to obtain the desired splitting; note that it typically depends
on $\ell$. 
\end{proof}

The decomposition theorem also holds for the underlying filtered $\Dmod$-modules,
because of the strictness of the complex $\fp(\Mmod, F_{\bullet} \Mmod)$.

\begin{corollary} \label{cor:decomposition-MF}
Under the same assumptions as above, let $(\Mmod, F_{\bullet} \Mmod)$ be a filtered
$\Dmod$-module that underlies a polarizable Hodge module. Then
\[
	\fp(\Mmod, F_{\bullet} \Mmod) 
		\simeq \bigoplus_{i \in \ZZ} 
			\bigl( \shH^i \fp(\Mmod, F_{\bullet} \Mmod) \bigr) \decal{-i}
\]
in the derived category of filtered $\Dmod$-modules on $Y$.
\end{corollary}

\subsection{Proof of the direct image theorem}

Both \theoremref{thm:structure} and \theoremref{thm:direct-image} are ultimately
consequences of results about polarized variations of Hodge structure. In the case of
the structure theorem, the main input are the results of Cattani, Kaplan, and Schmid
\cite{Schmid,CKS,CKS-L2} and Kashiware and Kawai \cite{KK} about degenerating polarized
variations of Hodge on the complement of a normal crossing divisor. In the case of
the direct image theorem, the main input are the results of Zucker \cite{Zucker}
about the cohomology of a polarized variation of Hodge structure
on a punctured Riemann surface. Note that the classical theory only deals with
objects that are mildly singular; Saito's theory can thus be seen as a mechanism that
reduces problems about polarizable variations of Hodge structure with arbitrary
singularities to ones with mild singularities.

Let me begin by explaining the proof of the direct image theorem
\cite[\S5.3]{Saito-HM}. Given the recursive nature of the definitions, it is clear
that \theoremref{thm:direct-image} has to be proved by induction on $\dim
Z$. The argument has three parts:
\begin{enumerate}
\item Establish the theorem for $\dim X = 1$.
\item Prove the theorem in the case $\dim f(Z) \geq 1$.
\item Prove the theorem in the case $\dim f(Z) = 0$.
\end{enumerate}

In the third part, we shall allow ourselves to use the structure theorem; one can
avoid this by using some ad-hoc arguments, but it simplifies the presentation. This
may look like a circular argument, because the proof of the structure theorem relies
on the direct image theorem -- but in fact it is not: the direct image theorem is only
used for resolutions of singularities, which fall into the case $\dim f(Z)
\geq 1$. Although I decided to keep the two proofs separate, one can actually make
everything work out by proving both theorems together by induction on $\dim Z$.

\paragraph{Part 1}
Saito first proves \theoremref{thm:direct-image} for the mapping from a compact
Riemann surface to a point; in fact, only the case $X = \PP^1$ is
needed. Let $M \in \HMZp{Z}{X}{w}$ be a polarized Hodge module with strict
support $Z \subseteq X$. Then $Z$ is either a point, in which case $M$ comes from a
polarized Hodge structure of weight $w$; or $Z$ is equal to $X$, in which case $M$
comes from a polarized variation of Hodge structure of weight $w-1$ on the complement
of finitely many points $x_1, \dotsc, x_m$. The first case is trivial; in the second
case, the direct image theorem follows from Zucker's work. In \cite{Zucker}, which
predates the theory of Hodge modules, Zucker proves that the $L^2$-cohomology groups
of the variation of Hodge structure carry a polarized Hodge structure. Although he
does not explicitly mention $\Dmod$-modules, he also proves that the filtered
$L^2$-complex of the variation of Hodge structure is quasi-isomorphic to the filtered
de Rham complex of $(\Mmod, F_{\bullet} \Mmod)$. The key point is that the asymptotic
behavior of the Hodge norm near each $x_i$ is controlled by the $V\!$-filtration on
$\Mmod$ and by the weight filtration on the nearby cycles; this is a consequence of
Schmid's results \cite{Schmid}. Zucker's construction therefore leads to the same
filtration and to the same pairing as in Saito's theory, and hence it
implies the direct image theorem. The interested reader can find a detailed
account of the proof in \cite[Chapter~3]{Sabbah}.

\paragraph{Part 2}
The next step is to prove the direct image theorem in the case $\dim f(Z) \geq 1$. By
induction, we can assume that we already know the result for all polarized Hodge
modules with strict support of dimension less than $\dim Z$. The key point in the
case $\dim f(Z) \geq 1$ is that applying a nearby or vanishing cycle functor on $Y$
will reduce the dimension of $Z$, and can therefore be handled by induction. 

To show that each $\shH^i \fl M$ is a Hodge module of weight $w+i$, we have to verify
the conditions in the definition; the same goes for showing that the polarization on
$M$ induces a polarization on the primitive part $\Pell \shH^{-i} \fl M = \ker
\ell^{i+1}$. This involves applying the functors $\psi_g$ and $\phi_{g,1}$, where $g$
is a locally defined holomorphic function on $Y$. Some care is needed because $\shH^i
\fl M$ is a priori not a filtered regular holonomic $\Dmod$-module with
$\QQ$-structure, but only lives in a larger abelian category (see
\parref{par:direct-image}).

Set $h = g \circ f$, and consider first the case where $f(Z) \not\subseteq
g^{-1}(0)$, or equivalently, $Z \not\subseteq h^{-1}(0)$. In the abelian category
mentioned above, one has a spectral sequence
\[
	E_1^{p,q} = \shH^{p+q} \fl \bigl( \gr_{-p}^W \psi_h M \bigr)
		\Longrightarrow \shH^{p+q} \fl \psi_h M.
\]
Each $\gr_{-p}^W \psi_h M$ is a polarized Hodge module of weight $w-1-p$ on
$h^{-1}(0) \cap Z$; by induction, $E_1^{p,q}$ is therefore a polarized Hodge module
of weight $w-1+q$. Note that we only have a polarization on the primitive part of
$\gr_{-p}^W \psi_h M$ with respect to $N$, and that \theoremref{thm:direct-image}
only produces a polarization on the primitive part with respect to $\ell$; this
makes it necessary to study the simultaneous action of $\ell$ and $N$. 
In any case, it follows that the entire spectral sequence is taking place in the
category of Hodge modules, and so it degenerates at $E_2$ for weight reasons. Now
some general results about the compatibility of the direct image functor with nearby
and vanishing cycles imply that
\begin{equation} \label{eq:psi-gh}
	\psi_g \shH^i \fl M \simeq \shH^i \fl \psi_h M \quad \text{and} \quad
		\phi_{g,1} \shH^i \fl M \simeq \shH^i \fl \phi_{h,1} M.
\end{equation}
This already shows that the objects on the left-hand side are Hodge modules on $Y$.
Another key point is that the filtration induced by the spectral sequence is the
monodromy filtration for the action of $N$ on $\psi_g \shH^i \fl M$; this is proved
by showing that the Lefschetz property for $(\ell, N)$ on the $E_1$-page of the
spectral sequence continues to hold on the $E_2$-page.

At this point, we can finally prove that the complex $\fp(\Mmod, F_{\bullet} \Mmod)$
is strict, and hence that each $\shH^i \fl M$ is a filtered regular holonomic
$\Dmod$-module with $\QQ$-structure. By controlling $\psi_g M$ and $\phi_{g,1} M$ in
the case when $f(Z) \not\subseteq g^{-1}(0)$, we actually know the whole associated
graded of $\Mmod$ with respect to the $V\!$-filtration. In sufficiently negative
degrees, the $V\!$-filtration is equivalent to the $g$-adic filtration, and so we
obtain information about $\fp(\Mmod, F_{\bullet} \Mmod)$ in a small analytic 
neighborhood of $g^{-1}(0)$. Since $\dim f(Z) \geq 1$, such neighborhoods cover
$f(Z)$; this shows that $\fp(\Mmod, F_{\bullet} \Mmod)$ is strict on all of $Y$.
Similar reasoning is used to prove the Lefschetz isomorphism.

Now we can start checking the conditions in the definition. Since $(\Mmod, F_{\bullet}
\Mmod)$ is quasi-unipotent and regular along $h = 0$, the arguments leading to
\eqref{eq:psi-gh} show that $\shH^i \fp(\Mmod, F_{\bullet} \Mmod)$ is quasi-unipotent
and regular along $g = 0$, too. This proves \eqref{eq:HM-def-1} in the
case $f(Z) \not\subseteq g^{-1}(0)$; the case $f(Z) \subseteq g^{-1}(0)$ is
more or less trivial. The recursive condition in \eqref{eq:HM-def-4}, as well as the
assertions about the polarization, then follow from the analysis of the spectral
sequence above.

The only remaining condition is \eqref{eq:HM-def-3}, namely that $\shH^i \fl M$ admits a
decomposition by strict support. Somewhat surprisingly, the proof of this fact uses
the polarization in an essential way. The idea is that the criterion in
\theoremref{thm:SS-criterion} reduces the problem to an identity among Hodge modules,
which can be checked pointwise after decomposing by strict support; in the end, it
becomes a linear algebra problem about certain families of polarized Hodge
structures. (The same result is used again during the proof of the structure theorem.)

\paragraph{Part 3}
It remains to prove \theoremref{thm:direct-image} in the case $\dim f(Z) = 0$. We may
clearly assume that $X$ is projective space and $Y$ is a point; now the idea is to
use a pencil of hyperplane sections to get into a situation where one can apply the
inductive hypothesis. Let $\pi \colon \Xt \to X$ be the blowup of $X$ at a generic
linear subspace of codimension $2$; as shown in the diagram
\[
\begin{tikzcd}[column sep=large,row sep=large]
\Xt \dar{p} \arrow{dr}{\ft} \rar{\pi} & X \dar{f} \\
\PP^1 \rar & \pt
\end{tikzcd}
\]
we also get a new morphism $p \colon \Xt \to \PP^1$. To apply the inductive
hypothesis, we need to construct a polarizable Hodge module $\Mt$ on $\Xt$; this can
easily be done with the help of the structure theorem. Let $\Zt \subseteq
\pi^{-1}(Z)$ be the irreducible component birational to $Z$. From
\theoremref{thm:structure}, we get an equivalence of categories
\[
	\HMZp{Z}{X}{w} \simeq \HMZp{\Zt}{\Xt}{w},
\]
and so $M$ determines a Hodge module $\Mt \in \HMZp{\Zt}{\Xt}{w}$ with a
polarization. We already know that the direct image theorem is true for $\Mt$ and
$\pi$, and so $\shH^0 \pil \Mt$ is a polarizable Hodge module of weight $w$; it is
clear from the construction that the summand with strict support $X$ is isomorphic to
$M$ (with the original polarization).

Now we apply the direct image theorem to $\Mt$ and $p \colon \Xt \to \PP^1$, and then
again to the Hodge modules $\shH^i \pl \Mt$ on $\PP^1$; the result is that the complex
$\ftp (\Mmodt, F_{\bullet} \Mmodt)$ is strict, and that $H^i(\Xt, \Mt) = \shH^i \ftl
\Mt$ is a Hodge structure of weight $w+i$. Both properties are inherited by direct
summands, and so it follows that $\fp(\Mmod, F_{\bullet} \Mmod)$ is also strict, and
that $H^i(X, M) = \shH^i \fl M$ is also a Hodge structure of weight $w+i$.

The remainder of the argument consists in proving the Lefschetz isomorphism and the
assertion about the polarization. Here Saito's method is to use the weak Lefschetz
theorem as much as possible, rather than trying to compare Lefschetz operators and
polarizations for $M$ and $\Mt$ directly; the point is that $\ell \simeq \il \circ
\iu$, where $i \colon Y \into X$ is the inclusion of a general hyperplane. This takes
care of all cases except for proving that the polarization on $M$ induces a
polarization of the Hodge structure $\Pell H^0(X, M) = \ker \ell$. In this remaining
case, one has
\[
	\Pell H^0(X, M) \into H^0 \bigl( \PP^1, \Pellt \shH^0 \pl \Mt \bigr),
\]
with a corresponding identity for the polarizations; since we control the
polarization on $\Pellt \shH^0 \pl \Mt$ by induction, an appeal to the theorem on
$\PP^1$ finishes the proof.

\subsection{Proof of the structure theorem}
\label{par:structure-proof}

Now let me explain how Saito proves the difficult half \theoremref{thm:structure},
namely that a generically defined polarized variation of Hodge structure on $Z$ of
weight $w - \dim Z$ extends uniquely to an object of $\HMZp{Z}{X}{w}$. Since it is
easy to deduce from the definition that there can be at most one extension, we shall
concentrate on the problem of constructing it. 

We first consider a simplified local version of the problem to which we can apply the
theory of degenerating variations of Hodge structure. Let $X = \Delta^n$ be a product
of disks, with coordinates $t_1, \dotsc, t_n$, and let $(\shV, F^{\bullet} \shV, V,
Q)$ be a polarized variation of Hodge structure on the complement of the divisor $t_1
\dotsb t_n = 0$. We shall do the following:
\begin{enumerate}
\item Extend the variation of Hodge structure to a filtered regular holonomic
$\Dmod$-module with $\QQ$-structure $M$, and construct a candidate for a
polarization.  
\item Show that $M$ is of normal crossing type, and that many of its properties are
therefore determined by combinatorial data.
\item Verify the conditions in the special case of a monomial function $g$.
\end{enumerate}
Afterwards, we can use resolution of singularities and the direct image
theorem to show that this special case is enough to solve the problem in general.

\paragraph{Part 1}
For the time being, we assume that we have a polarized variation of Hodge structure of
weight $w-n$, defined on the complement of the divisor $t_1 \dotsm t_n = 0$ in
$X = \Delta^n$, with quasi-unipotent local monodromy. We have to construct an
extension $M = (\Mmod, F_{\bullet} \Mmod, K)$ to a 
filtered regular holonomic $\Dmod$-module with $\QQ$-structure. Define $K$ as the
intersection complex of the local system $V$, and $\Mmod$ as the corresponding
regular holonomic $\Dmod$-module; both have strict support $X$. Since
$(\Mmod, F_{\bullet} \Mmod)$ is supposed to be quasi-unipotent and regular
along $t_1 \dotsm t_n = 0$, we are forced to define the filtration $F_{\bullet}
\Mmod$ as in \eqref{eq:F-formula}. More concretely,
\[
	\Mmod = \Bigl( \omX \tensor \shVt^{(-1,0]} \Bigr) \cdot \Dmod_X 
		\subseteq \omX \tensor \shVt
\]
is the $\Dmod$-submodule generated by the lattice $\shVt^{(-1,0]}$ in Deligne's
meromorphic extension $\shVt$ of the flat bundle $(\shV, \nabla)$, and the filtration
is given by the formula
\[
	F_p \Mmod = \sum_{i = 0}^{\infty} 
		\Bigl( \omX \tensor F^{i-p-n} \shVt^{(-1,0]} \Bigr) \cdot F_i \Dmod_X.
\]
From the polarization on $V$, one can also construct a morphism $K(w) \to \DD K$ that
is a candidate for being a polarization on $M$. 

\paragraph{Part 2}
To prove that $M$ is a polarized Hodge module, we use induction on $n \geq 0$. Here
it is useful to work inside a larger class of objects that is preserved by various
operations. From the nilpotent orbit theorem \cite{Schmid}, one can deduce that $M$
has the following properties:
\begin{aenumerate}
\item The perverse sheaf $K$ is quasi-unipotent, and the natural stratification
of $\Delta^n$ is adapted to $K$.
\item The $n+1$ filtrations $F_{\bullet} \Mmod, V_{\bullet}^1 \Mmod, \dotsc, V_{\bullet}^n
\Mmod$ are compatible, where $V_{\bullet}^i \Mmod$ denotes the $V\!$-filtration with
respect to the function $t_i$.
\item Set $\partial_i = \partial/\partial t_i$ for every $i = 1, \dotsc, n$; then one has
\begin{align*}
	F_p V_{\alpha}^i \Mmod \cdot t_i &= F_p V_{\alpha-1}^i \Mmod 
		\quad \text{for $\alpha < 0$,} \\
	F_p \gr_{\alpha}^{V^i} \Mmod \cdot \partial_i &= F_{p+1} \gr_{\alpha+1}^{V^i} \Mmod
		\quad \text{for $\alpha > -1$,}
\end{align*}
which is a special case of the conditions in \definitionref{def:strict-spec}.
\end{aenumerate}
In general, we say that an object with these properties is of \define{normal crossing
type}. 

This definition is closely related to the combinatorial description of regular holonomic
$\Dmod$-modules \cite{GGM}. If the first condition is satisfied,
then $\Mmod$ is determined by the following \define{combinatorial data}: the
collection of vector spaces
\[
	\gr_{\alpha}^V \Mmod = \gr_{\alpha_1}^{V^1} \dotsb \gr_{\alpha_n}^{V^n} \Mmod,
		\quad \text{for $\alpha \in (\QQ \cap [-1, 0])^n$,}
\]
and the linear mappings induced by $t_1, \dotsc, t_n$ and $\partial_1, \dotsc,
\partial_n$ and $N_1, \dotsc, N_n$. As in \definitionref{def:strict-spec}, Saito
introduces the other two conditions in order to deal with the filtration $F_{\bullet}
\Mmod$ (which is not itself combinatorial). For objects of normal crossing type, the
combinatorial data controls the properties of morphisms: for example, if $\varphi
\colon M_1 \to M_2$ is a morphism between objects of normal crossing type, then $\ker
\varphi$ and $\im \varphi$ are again of normal crossing type, provided that the
induced morphisms on combinatorial data are strictly compatible with the Hodge
filtration.

Saito also shows that if $M$ is of normal crossing type, then the filtered
$\Dmod$-module $(\Mmod, F_{\bullet} \Mmod)$ is Cohen-Macaulay, and quasi-unipotent
and regular along any monomial $g = t_1^{m_1} \dotsm t_n^{m_n}$;
moreover, the vanishing cycles $\psi_g M$ are again of normal crossing type. Part of 
the argument consists in finding formulas for the combinatorial data of $\psi_g M$, and
for the nilpotent operator $N$, in terms of the combinatorial data of $M$ itself.

\paragraph{Part 3}
Now we can verify all the conditions in the case of a monomial $g = t_1^{m_1}
\dotsm t_n^{m_n}$. Since we got $M$ from a polarized variation of Hodge
structure, the combinatorial data lives in the category of mixed Hodge
structures: for nearby cycles, this follows from the $\SL(2)$-orbit theorem
\cite{CKS}, and for vanishing cycles from the ``vanishing cycle theorem'' \cite{KK}
respectively ``descent lemma'' \cite{CKS-L2}. In particular, $t_1, \dotsc, t_n$ and
$\partial_1, \dotsc, \partial_n$ are morphisms of mixed Hodge structure, and
therefore strictly compatible with the Hodge filtration.

Since $M$ is of normal crossing type, $(\Mmod, F_{\bullet} \Mmod)$ is quasi-unipotent
and regular along $g = 0$, and $\psi_g M$ is again an object of normal crossing type.
From the combinatorial data for $M$, one can write down that for $\psi_g M$, and 
show that it also lives in the category of mixed Hodge structures. Because $N$ is a
morphism of mixed Hodge structures, and therefore strictly compatible with the Hodge
filtration, it follows that the monodromy filtration $W(N)_{\bullet} \psi_g M$, as well
as the primitive decomposition for $N$, are also of normal crossing type. In particular,
the primitive part $P_N \gr_{\ell}^{W(N)} \psi_g M$ is again an object of normal crossing
type.

The next step is to prove that $P_N \gr_{\ell}^{W(N)} \psi_g M$ admits a decomposition by
strict support. It is enough to check the condition in \theoremref{thm:SS-criterion}
for the functions $t_1, \dotsc, t_n$; because we are dealing with objects of normal
crossing type, the problem further reduces to a statement about polarized Hodge
structures (which has already been used during the proof of the direct image theorem).
Note that every summand in the decomposition by strict support is again of normal
crossing type.

Now let $M'$ be any summand in the decomposition; after permuting
the coordinates, its support will be of the form $\Delta^m \times \{0\} \subseteq
\Delta^n$. To conclude the argument, we have to show that $M'$ is again an
extension of a polarized variation of Hodge structure: by induction, this will
guarantee that $M'$ is a Hodge module of weight $w - 1 + \ell$, and that the induced
morphism $K'(w - 1 + \ell) \to \DD K'$ is a polarization. Because $M'$ is an object
of normal crossing type, $(\Mmod', F_{\bullet} \Mmod')$ is quasi-unipotent and
regular along $t_1 \dotsm t_m = 0$, and so $M'$ must be the extension of a variation of
Hodge structure $(\shV', F^{\bullet} \shV', V')$; the problem is to show that the
induced bilinear form $Q'$ on $V'$ is a polarization. Here Saito appeals to the
$\SL(2)$-orbit theorem and its consequences \cite{CKS}: by looking at the
combinatorial data, one sees that $\psi_{t_1} \dotsm \psi_{t_m} M'$ is a nilpotent
orbit; this implies that $Q'$ polarizes the variation of Hodge structure in a small
neighborhood of the origin in $\Delta^m$.

\begin{proof}[Proof of \theoremref{thm:structure}]
Suppose first that $V$ is defined on the complement of a normal crossing divisor $D$
in a complex manifold $X$. The same construction as above produces an extension $M$
to a filtered regular holonomic $\Dmod$-module with $\QQ$-structure and strict
support $X$. Now we have to check that all the conditions are satisfied for an arbitrary
locally defined holomorphic function $g$. Resolution of singularities, together with
some technical results from the proof of the direct image theorem, reduces this to
the case where $X$ is a product of disks and $g$ is a monomial; since we have already
proved the result in that case, we conclude that $M$ is a polarized Hodge module.

To deal with the general case, we again use resolution of singularities and the
direct image theorem. Choose an embedded resolution of singularities $f \colon \Zt
\to X$ that is an isomorphism over the subset of $Z$ where the variation of Hodge
structure is defined, and that makes the complement into a normal crossing divisor.
We already know that we can extend the variation of Hodge structure to an object $\Mt
\in \HMZp{\Zt}{\Zt}{w}$; since we can choose $f$ to be projective, we can then apply
\theoremref{thm:direct-image} to conclude that $\shH^0 \fl \Mt \in \HMp{X}{w}$. Now
the direct summand with strict support $Z$ is the desired extension.
\end{proof}

\section{Mixed Hodge modules}

The cohomology groups of algebraic varieties that are not smooth and projective still
carry mixed Hodge structures. To include those cases into his theory, Saito was
led to consider mixed objects. In fact, the definition of pure Hodge modules already
hints at the existence of a more general theory: one of the conditions says that
$\gr_{\ell}^{W(N)}(\psi_f M)$ is a Hodge module of weight $w-1+\ell$ for every $\ell
\in \ZZ$; this suggests that the nearby cycles $\psi_f M$, together with the
monodromy filtration for $N$ shifted by $w-1$ steps, should actually be a ``mixed''
Hodge module.  

\subsection{Weakly mixed Hodge modules}
\label{par:MHW}

To define mixed Hodge modules, Saito first introduces an auxiliary category $\MHW{X}$ of
\define{weakly mixed Hodge modules}. Its objects are pairs $(M, W_{\bullet} M)$,
where $M$ is a filtered regular holonomic $\Dmod$-module with $\QQ$-structure, and 
$W_{\bullet} M$ is a finite increasing filtration with the property that
\[
	\gr_{\ell}^W \! M \in \HM{X}{\ell}.
\]
A weakly mixed Hodge module is called \define{graded-polarizable} if the individual
Hodge modules $\gr_{\ell}^W \! M$ are polarizable; we denote this category by
the symbol $\MHWp{X}$. 

Certain results from the pure case carry over to this setting without additional
effort. For example, suppose that $(M, W_{\bullet} M) \in \MHWp{X}$, and that we have
a projective morphism $f \colon X \to Y$. Then the spectral sequence
\[
	E_1^{p,q} = \shH^{p+q} \fl \bigl( \gr_{-p}^W M \bigr) 
		\Longrightarrow \shH^{p+q} \fl M
\]
degenerates at $E_2$, and each $\shH^i \fl M$, together with the filtration induced
by the spectral sequence, again belongs to $\MHWp{Y}$. This is an easy
consequence of \theoremref{thm:direct-image} and the fact that there are no
nontrivial morphisms between polarizable Hodge modules of different weights. One can also
show that every morphism in the category $\MHW{X}$ is strictly compatible with the
filtrations $F_{\bullet} \Mmod$ and $W_{\bullet} \Mmod$ (in the strong sense).

In order to have a satisfactory theory of mixed Hodge modules, however, we will need
to impose restrictions on how the individual pure Hodge modules can be put together.
In fact, the same problem already appears in the study of variations of mixed Hodge
structure. The work of Cattani, Kaplan, and Schmid \cite{CKS} shows that every
polarizable variation of Hodge structure degenerates in a controlled way; but for
variations of mixed Hodge structure, this is no longer the case. To get a theory
similar to the pure case, one has to restrict to \define{admissible} variations of
mixed Hodge structure; this notion was introduced by Steenbrink and Zucker \cite{SZ},
and further developed by Kashiwara \cite{Kashiwara-st}.

\begin{example}
When we assemble a mixed Hodge structure from $\ZZ(0)$ and $\ZZ(1)$, the
result is determined by a non-zero complex number, because
\[
	\Ext_{\MHS}^1 \bigl( \ZZ(0), \ZZ(1) \bigr) \simeq \CCst.
\]
A variation of mixed Hodge structure of this type, say on the punctured disk,
is therefore the same thing as a holomorphic function $f \colon \dst \to \CCst$;
note that $f$ may have an essential singularity at the origin. In this simple
example, the admissibility condition amounts to allowing only meromorphic functions $f$.
\end{example}

Note that admissibility is always defined relative to a partial compactification;
two partial compactifications that are not bimeromorphically equivalent may lead to
different notions of admissibility. Only in the case of algebraic varieties,
where any two compactifications are automatically birationally equivalent, can we
speak of admissibility without specifying the compactification.

\subsection{Definition of mixed Hodge modules}
\label{par:def-MHM}

In this section, we present a simplified definition of the category of mixed
Hodge modules that Saito has developed over the years.%
\footnote{I learned the details of this new definition from a lecture by Mochizuki.}
The point is that there is some redundancy in the original definition
\cite[\S2.d]{Saito-MHM}, which can be eliminated by a systematic use of the stability
of $\MHMp{X}$ under taking subquotients.

We shall only define mixed Hodge modules on complex manifolds; more general analytic
spaces are again dealt with by local embeddings into complex manifolds. Let $(M,
W_{\bullet} M)$ be a weakly mixed Hodge module on a complex manifold $X$. We have to
decide when $M$ should be called a mixed Hodge module; the idea is to adapt the
recursive definition from the pure case, by also imposing a condition similar to
admissibility in terms of nearby and vanishing cycle functors. 

Suppose then that $f \colon X \to \CC$ is a non-constant holomorphic function on $X$. By
induction on the length of the weight filtration, one can easily show that the
underlying filtered $\Dmod$-module $(\Mmod, F_{\bullet} \Mmod)$ is quasi-unipotent
and regular along $f = 0$; recall that this means the existence of a rational
$V\!$-filtration $V_{\bullet} \Mmodf$ that interacts well with the Hodge filtration
$F_{\bullet} \Mmodf$. Together with the weight filtration, we now have three
filtrations on $\Mmodf$, and so we require that
\begin{equation} \label{eq:MHM-def-1}
	\text{the three filtrations $F_{\bullet} \Mmodf$, $V_{\bullet} \Mmodf$, and
		$W_{\bullet} \Mmodf$ are compatible.}
\end{equation}
Roughly speaking, compatibility means that when we compute the associated graded, the
order of the three filtrations does not matter; this is automatic in the case of two
filtrations, but not usually true in the case of three or more.%
\footnote{This problem has been studied at length in \cite[\S1]{Saito-MHM}.}

Assuming the condition above, the nearby and vanishing cycles of each $W_i M$ are
defined, and so we can introduce the naive limit filtrations
\[
	L_i(\psi_f M) = \psi_f \bigl( W_{i+1} M \bigr) 
		\quad \text{and} \quad
	L_i(\phi_{f,1} M) = \phi_{f,1} \bigl( W_i M \bigr).
\]
Now recall that both $\psi_f M$ and $\phi_{f,1} M$ are equipped with a nilpotent
endomorphism $N = (2\pi i)^{-1} \log T_u$; since $N$ preserves the weight
filtration, it also preserves the two limit filtrations. As in the admissibility
condition of Steenbrink and Zucker, we ask for the existence of the so-called
\define{relative monodromy filtration}: 
\begin{equation} \label{eq:MHM-def-2}
	\text{\parbox{10.2cm}{%
		the relative monodromy filtrations $W_{\bullet}(\psi_f M) 
			= W_{\bullet} \bigl( N, L_{\bullet}(\psi_f M) \bigr)$ \\
		and $W_{\bullet}(\phi_{f,1} M) = W_{\bullet} \bigl( N, 
		L_{\bullet}(\phi_{f,1} M) \bigr)$ for the action of $N$ exist.}}
\end{equation}
This is the natural replacement for the monodromy filtration of $N$ in the pure case;
but while the existence of the monodromy filtration is automatic, the existence of
the relative monodromy filtration is a nontrivial requirement. We say that $(M, W_{\bullet}
M)$ is \define{admissible along $f = 0$} if both of the conditions in
\eqref{eq:MHM-def-1} and \eqref{eq:MHM-def-2} are satisfied; one should think of this
as saying that the restriction of $(M, W_{\bullet} M)$ to the open subset $X
\setminus f^{-1}(0)$ is admissible relative to $X$.

\begin{exercise}
Consider a weakly mixed Hodge module $(M, W_{\bullet} M)$ on the disk $\Delta$,
whose restriction to $\dst$ is a variation of mixed Hodge structure with unipotent
monodromy. Show that $(M, W_{\bullet} M)$ is admissible along $t = 0$ if and only if
the variation is admissible in the sense of Steenbrink and Zucker \cite[\S3.13]{SZ}.
\end{exercise}

Now we can define the category of mixed Hodge modules; as in the pure case, the
definition is recursive. A weakly mixed Hodge module $(M, W_{\bullet} M) \in \MHW{X}$
is called a \define{mixed Hodge module} if, for every locally defined
holomorphic function $f \colon U \to \CC$, the pair $(M, W_{\bullet} M)$ is
admissible along $f = 0$, and 
\begin{equation} \label{eq:MHM-def-3}
	\text{$\bigl( \psi_f M, W_{\bullet}(\psi_f M) \bigr)$ and 
		$\bigl( \phi_{f,1} M, W_{\bullet}(\phi_{f,1} M) \bigr)$ are mixed Hodge modules}
\end{equation}
whenever $f^{-1}(0)$ does not contain any irreducible components of $U \cap \Supp
\Mmod$. This definition makes sense because $\psi_f M$ and $\phi_{f,1} M$ are then
supported in a subset of strictly smaller dimension. Note that the shifts by $w-1$
and $w$ from the pure case are now built into the definition of the relative
monodromy filtration.

\begin{definition}
We denote by 
\[
	\MHM{X} \subseteq \MHW{X} \quad \text{and} \quad
		\MHMp{X} = \MHM{X} \cap \MHWp{X}
\]
the full subcategories of all (graded-polarizable) mixed Hodge modules. 
\end{definition}

A morphism between two Hodge modules is simply a morphism between the underlying
filtered regular holonomic $\Dmod$-modules with $\QQ$-structure that is compatible
with the weight filtrations. As in the case of $\MHW{X}$, one can show that every
morphism is strictly compatible with the filtrations $F_{\bullet} \Mmod$ and
$W_{\bullet} \Mmod$ (in the strong sense).

\subsection{Properties of mixed Hodge modules}

In this section, we collect a few important properties of mixed Hodge modules. First,
$\MHM{X}$ and $\MHMp{X}$ are abelian categories; here the main point is that every
subquotient of a mixed Hodge module in the category $\MHW{X}$ is again a
mixed Hodge module. Almost by definition, (graded-polarizable) mixed Hodge modules
are stable under the application of nearby and vanishing cycle functors. Extending
the other standard functors from perverse sheaves to mixed Hodge modules is more
involved. Given a projective morphism $f \colon X \to Y$, one has a collection of
cohomological functors
\[
	\shH^i \fl \colon \MHMp{X} \to \MHMp{Y};
\]
see the discussion near the beginning of \parref{par:MHW}. As explained in
\parref{par:inverse-image}, one can also define the cohomological inverse image functors
\[
	\shH^i \fu \colon \MHMp{X} \to \MHMp{Y} \quad \text{and} \quad 
		\shH^i \fus \colon \MHMp{X} \to \MHMp{Y}
\]
for an arbitrary morphism $f \colon Y \to X$. Lastly, one can define a duality
functor 
\[
	\DD \colon \MHMp{X} \to \MHMp{X}^{\opp}, 
\]
compatible with Verdier duality for the underlying perverse sheaves (see
\parref{par:duality}).

There is also a version of \theoremref{thm:structure}, relating mixed Hodge modules
and admissible variations of mixed Hodge structure \cite[Theorem~3.27]{Saito-MHM}.
Suppose that $(M, W_{\bullet}
M)$ is a graded-polarizable mixed Hodge module on a complex manifold $X$, and let
$Z \subseteq X$ be an irreducible component of the support of $M$. After restricting
to a suitable Zariski-open subset of $Z$, we obtain a graded-polarizable variation
of mixed Hodge structure; it is not hard to deduce from the definition that it must be
admissible relative to $Z$. The converse is also true, but much harder to prove.

\begin{theorem} \label{thm:extendable}
Let $X$ be a complex manifold, and $Z \subseteq X$ an irreducible closed analytic
subvariety of $X$. A graded-polarizable variation of mixed Hodge structure on a
Zariski-open subset of $Z$ can be extended to a mixed Hodge module on $X$ if and only
if it is admissible relative to $Z$.
\end{theorem}

Another important result is the description of mixed Hodge modules by a ``gluing''
procedure, similar to the case of perverse sheaves \cite{Verdier,Beilinson}. Suppose
that $X$ is a complex manifold, and $Z \subseteq X$ a closed analytic subvariety; we
denote by $j
\colon X \setminus Z \into X$ the inclusion of the open complement. Saito shows that
a mixed Hodge module $M \in \MHMp{X}$ can be reconstructed from the following
information: its restriction to $X \setminus Z$; a mixed Hodge module on $Z$; and
some gluing data. Let me explain how this works in the special case where
$Z = f^{-1}(0)$ is the zero locus of a holomorphic function $f \colon X \to \CC$.
Here one can show that $M \in \MHMp{X}$ is uniquely determined by the quadruple
\[
	\bigl( j^{-1} M, \phi_{f,1} M, \can, \var \bigr).
\]
Recall that $\can \colon \psi_{f,1} M \to \phi_{f,1} M$ and $\var \colon \phi_{f,1} M \to
\psi_{f,1} M(-1)$ are morphisms of mixed Hodge modules, and that the unipotent nearby
cycles $\psi_{f,1} M$ only depend on $j^{-1} M$. This suggests considering all
quadruples of the form
\[
	\bigl( M', M'', u, v \bigr),
\]
where $M' \in \MHMp{X \setminus Z}$ is \define{extendable} to $X$; where $M'' \in
\MHMp{Z}$; and where
\[
	u \colon \psi_{f,1} M' \to M'' \quad \text{and} \quad
		v \colon M'' \to \psi_{f,1} M'(-1)
\]
are morphisms of mixed Hodge modules with the property that $v \circ u = N$.
We denote the category of such quadruples (with the obvious
morphisms) by the symbol $\MHMpex{X \setminus Z}{Z}$. One has the following result
\cite[Proposition~2.28]{Saito-MHM}.

\begin{theorem}
The functor
\[
	\MHMp{X} \to \MHMpex{X \setminus Z}{Z}, \quad
		M \mapsto \bigl( j^{-1} M, \phi_{f,1} M, \can, \var \bigr),
\]
is an equivalence of categories.
\end{theorem}

The proof is essentially the same as in the case of perverse sheaves
\cite{Beilinson}, and makes use of Beilinson's maximal extension functor. There is a
similar theorem for more general $Z \subseteq X$, in terms of Verdier specialization.
This result can be used to define the two functors
\[
	\jl j^{-1} \colon \MHMp{X} \to \MHMp{X} \quad \text{and} \quad
		\jls j^{-1} \colon \MHMp{X} \to \MHMp{X},
\]
in a way that is compatible with the corresponding functors for perverse sheaves.
As the notation suggests, $\jl j^{-1} M$ and $\jls j^{-1} M$ only depend on the
restriction of $M$ to $X \setminus Z$, and can thus be considered as functors
from extendable mixed Hodge modules on $X \setminus Z$ to mixed Hodge modules on $X$.

\begin{example}
In the special case considered above, $\jl j^{-1} M$ is the mixed Hodge
module corresponding to the quadruple $\bigl( j^{-1} M, \psi_{f,1} M(-1), N, \id
\bigr)$, and $\jls j^{-1} M$ the one corresponding to $\bigl( j^{-1} M, \psi_{f,1} M,
\id, N \bigr)$.
\end{example}

Note that an arbitrary graded-polarizable mixed Hodge module on $X \setminus Z$ may
not be extendable to $X$: for graded-polarizable variations of mixed Hodge structure,
for example, extendability is equivalent to admissibility (by
\theoremref{thm:extendable}). This is why one cannot define $\jl$ and $\jls$ as
functors from $\MHMp{X \setminus Z}$ to $\MHMp{X}$.

\subsection{Algebraic mixed Hodge modules}

In the case where $X$ is a complex algebraic variety, it is more natural to consider
only algebraic mixed Hodge modules on $X$. Choose a \define{compactification} of $X$,
meaning a proper algebraic variety $\Xb$ containing $X$ as a dense Zariski-open
subset. Then according to Serre's G.A.G.A.~theorem, a coherent sheaf on the
associated analytic space $\Xan$ is algebraic if and only if it is the restriction of
a coherent sheaf from $\Xban$; this condition is independent of the choice of
$\Xb$. A similar result holds for coherent filtered $\Dmod$-modules.
This suggests the following definition.

\begin{definition} \label{def:MHM-algebraic}
Let $X$ be a complex algebraic variety. The category $\MHMalg{X}$ of 
\define{algebraic mixed Hodge modules} is defined as the image of $\MHMp{\Xban}$ under
the natural restriction functor $\MHMp{\Xban} \to \MHMp{\Xan}$.
\end{definition}

One can show that the resulting category is independent of the choice of $\Xb$; the
reason is that any two complete algebraic varieties containing $X$ as a dense
Zariski-open subset are birationally equivalent. When $M$ is an algebraic mixed Hodge
module, the perverse sheaf $\rat M$ on $\Xan$ is constructible with respect to an
algebraic stratification of $X$, and the coherent sheaves $F_{\bullet} \Mmod$ are
algebraic; this justifies the name. It is important to remember that algebraic mixed
Hodge modules are, by definition, polarizable and extendable.

\begin{exercise}
Use the structure theorem to prove that every polarizable variation of Hodge
structure on an algebraic variety $X$ is an object of $\MHMalg{X}$.
\end{exercise}

A drawback of the above definition is that it involves choosing a complete variety
containing $X$. Fortunately, Saito \cite{Saito-new} has recently published an
elegant intrinsic description of the category $\MHMalg{X}$. We shall state it as a
theorem here; in an alternative treatment of the theory, it could perhaps become
the definition.

\begin{theorem} \label{thm:simplified}
A weakly mixed Hodge module $(M, W_{\bullet} M) \in \MHW{\Xan}$ belongs to the category
$\MHMalg{X}$ if and only if $X$ can be covered by Zariski-open subsets $U$ with the
following four properties:
\begin{enumerate}
\item There exists $f \in \OX(U)$ such that $U \setminus f^{-1}(0)$ is smooth and dense
in $U$.
\item The restriction of $(M, W_{\bullet} M)$ to the open subset $\Uan \setminus f^{-1}(0)$
is a graded-polarizable admissible variation of mixed Hodge structure.
\item The pair $(M, W_{\bullet} M)$ is admissible along $f = 0$.
\item The pair $\bigl( \psi_f M, W_{\bullet}(\psi_f M) \bigr)$ belongs to 
$\MHMalgg{f^{-1}(0)}$.
\end{enumerate}
\end{theorem}

Note that if one takes the conditions in the theorem as the definition $\MHMalg{X}$,
then every graded-polarizable admissible variation of mixed Hodge structure on $X$ is
automatically an algebraic mixed Hodge module (because $f^{-1}(0)$ is allowed to be
empty). In fact, the point behind \theoremref{thm:simplified} is precisely that
every graded-polarizable admissible variation of mixed Hodge structure on $\Xan$ can be
extended to an object of $\MHMp{\Xban}$. 

\subsection{Derived categories and weights}

As mentioned in the introduction, Saito's theory is more satisfactory in the case of
algebraic varieties. The reason is that admissibility is independent of the
compactification (because any two compactifications are birationally
equivalent); and that one has direct image functors for arbitrary morphisms
(because every morphism can be factored into the composition of an open embedding and
a proper morphism).

We denote by $\Db \MHMalg{X}$ the bounded derived category of the abelian category
$\MHMalg{X}$; its objects are bounded complexes of algebraic mixed Hodge modules on
$X$. By construction, we have an exact functor
\[
	\rat \colon \Db \MHMalg{X} \to \Dbc(\QQ_{\Xan})
\]
to the bounded derived category of algebraically constructible complexes. The functor
associates to a mixed Hodge module the underlying perverse sheaf; it is
faithful, essentially because a morphism between two mixed Hodge modules is zero if
and only if the corresponding morphism on perverse sheaves is zero.

 
The advantage of working with algebraic varieties is that $\Db \MHMalg{X}$ satisfies 
the same six-functor formalism as in the case of constructible complexes \cite{BBD}.
Setting up this formalism is, however, not a trivial task. The easiest case is that
of the duality functor (see \parref{par:duality}): because
\[
	\DD \colon \MHMalg{X} \to \MHMalg{X}^{\opp}
\]
is exact, it extends without ado to the derived category. As explained in
\parref{par:inverse-image}, we have a collection of cohomological inverse image functors
\[
	\shH^j \fu \colon \MHMalg{Y} \to \MHMalg{X} \quad \text{and} \quad
		\shH^j \fus \colon \MHMalg{Y} \to \MHMalg{X}
\]
for an arbitrary morphism $f \colon X \to Y$ between algebraic varieties; the
functors $\fu$ and $\fus$ on the derived category are defined by working with
\v{C}ech complexes for suitable affine open coverings. One also has a collection of
cohomological direct image functors
\[
	\shH^j \fl \colon \MHMalg{X} \to \MHMalg{Y} \quad \text{and} \quad
		\shH^j \fls \colon \MHMalg{X} \to \MHMalg{Y}.
\]
In the case where $f$ is proper, this follows from \theoremref{thm:direct-image} and
Chow's lemma; in the general case, one factors $f$ into an open embedding $j$ followed by a
proper morphism, and uses the fact that the two functors $\jl$ and $\jls$ are
well-defined for algebraic mixed Hodge modules. The functors $\fl$ and $\fls$ are then again
constructed with the help of affine open coverings, following the method introduced
in \cite{Beilinson}. After all the functors have been constructed, one then has to
prove that they satisfy all the standard compatibility and adjointness relations.

\begin{example}
Once the whole theory is in place, one can use it to put mixed Hodge structures on
the cohomology groups of algebraic varieties. Let $f \colon X \to \pt$ denote the
morphism to a point; then one has a complex of algebraic mixed Hodge modules 
\[
	\QQ_X^H = \fu \QQ(0) \in \Db \MHMalg{X}. 
\]
When $X$ is smooth and $n$-dimensional, $\QQ_X^H \decal{n}$ is concentrated in degree
$0$, and is isomorphic to the pure Hodge module we were denoting by $\QQ_X^H
\decal{n}$ earlier. Now
\[
	H^i(X, \QQ_X^H) = H^i \fl \QQ_X^H \quad \text{and} \quad
		H_c^i(X, \QQ_X^H) = H^i \fls \QQ_X^H
\]
are graded-polarizable mixed Hodge structures on the $\QQ$-vector spaces $H^i(X,
\QQ)$ and $H_c^i(X, \QQ)$; with considerable effort, Saito managed to show that these
mixed Hodge structures are the same as the ones defined by Deligne \cite{Saito-MHC}.
\end{example}

Let me end this chapter by briefly discussing the weight formalism on the derived
category of algebraic mixed Hodge modules; it works in exactly the same way as in the
case of mixed complexes \cite[\S5.1.5]{BBD}.

\begin{definition}
We say that a complex $M \in \Db \MHMalg{X}$ is 
\begin{aenumerate}
\item \define{mixed of weight $\leq w$} if $\gr_i^W \shH^j(M) = 0$ for
$i > j + w$;
\item \define{mixed of weight $\geq w$} if $\gr_i^W \shH^j(M) = 0$ for
$i < j + w$;
\item \define{pure of weight $w$} if $\gr_i^W \shH^j(M) = 0$ for $i \neq j + w$.
\end{aenumerate}
\end{definition}

Saito shows that when $M$ is mixed of weight $\leq w$, both $\fls M$ and $\fu M$ are
again mixed of weight $\leq w$; when $M$ is mixed of weight $\geq w$, both $\fl M$
and $\fus M$ are again mixed of weight $\geq w$. It is also easy to see that $M$ is
mixed of weight $\leq w$ if and only if $\DD M$ is mixed of weight $\geq -w$. In
particular, pure complexes are stable under direct images by proper morphisms, and
under the duality functor. If $M_1$ is mixed of weight $\leq w_1$, and $M_2$ is mixed
of weight $\geq w_2$, then $\Ext^i(M_1, M_2) = 0$ for $i > w_1 - w_2$. This implies
formally that every pure complex splits into the direct sum of its cohomology
modules. The resulting non-canonical isomorphism
\[
	M \simeq \bigoplus_{j \in \ZZ} \shH^j(M) \decal{-j} \in \Db \MHMalg{X}
\]
gives another perspective on the decomposition theorem.

\section{Two applications}

The two results in this chapter are among the first applications of Saito's theory;
the interested reader can find more details in an article by Masahiko Saito
\cite{RIMS}.

\subsection{Saito's vanishing theorem}

In \cite[\S2.g]{Saito-MHM}, Saito proved the following vanishing theorem
for mixed Hodge modules. It contains many familiar results in algebraic geometry as
special cases: for example, if $X$ is an $n$-dimensional smooth projective variety,
and if we take $M = \QQ_X^H \decal{n}$, then 
\[
	\gr_{-p}^F \DR(\omX) \simeq \OmX^p \decal{n-p}, 
\]
and Saito's result specializes to the Kodaira-Nakano vanishing theorem.

\begin{theorem}[Vanishing Theorem] \label{thm:vanishing}
Suppose that $(\Mmod, F_{\bullet} \Mmod)$ underlies a mixed Hodge module $M$ on a
projective algebraic variety $X$. Then one has
\begin{align*}
	H^i \Bigl( X, \gr_k^F \DR(\Mmod) \tensor L \Bigr) = 0 \quad &\text{if $i>0$,} \\
	H^i \Bigl( X, \gr_k^F \DR(\Mmod) \tensor L^{-1} \Bigr) = 0 \quad &\text{if $i<0$,}
\end{align*}
where $L$ is an arbitrary ample line bundle on $X$.
\end{theorem}

Recall from the discussion in \parref{par:Kashiwara} that $(\Mmod, F_{\bullet}
\Mmod)$ lives on some ambient projective space; nevertheless, each
$\gr_k^F \DR(\Mmod)$ is a well-defined complex of coherent sheaves on $X$,
independent of the choice of embedding (by \exerciseref{ex:DR}).

The proof of \theoremref{thm:vanishing} is similar to Ramanujam's method for deducing
the Kodaira vanishing theorem from the weak Lefschetz theorem \cite{Ramanujam}.
Choose a sufficiently large integer $d \geq 2$ that makes $L^d$ very ample, and let
$i \colon X \into \PP^N$ denote the corresponding embedding into projective space; we
can assume that $M \in \MHM{\PP^N}$, with $\Supp M \subseteq X$. Let $H \subseteq
\PP^N$ be a generic hyperplane; then $H \cap X$ is the divisor of a section of $L^d$,
and therefore determines a branched covering $\pi \colon Y \to X$ of degree $d$ to
which one can pull back $M$. Then one can show that the $\Dmod$-module underlying
$\pil \piu M$ contains a summand isomorphic to $\Mmod(\ast H) \tensor L^{-1}$; this
implies the vanishing of
\[
	H^i \Bigl( X, \gr_k^F \DR \bigl( \Mmod(\ast H) \bigr) \tensor L^{-1} \Bigr)
\]
for $i \neq 0$. In fact, the formalism of mixed Hodge modules reduces this to Artin's
vanishing theorem: on the affine variety $Y \setminus \pi^{-1}(H \cap X)$, the
cohomology of a self-dual perverse sheaf vanishes in all degrees $\neq 0$. One can
now deduce the desired vanishing for $\gr_k^F \DR(\Mmod) \tensor L^{-1}$ from the
exact sequence
\[
	0 \to \Mmod \to \Mmod(\ast H) \to \ip \Bigl( \Mmod \restr{H} \Bigr) \to 0
\]
by induction on the dimension. A recent article by Popa \cite{Popa} explains the
details of the proof, as well as some applications of Saito's theorem to algebraic
geometry. An alternative proof of \theoremref{thm:vanishing} can be found in
\cite{saito-vanishing}.

\subsection{Koll\'ar's conjecture}
\label{par:Kollar}

In \cite{KollarI,KollarII}, Koll\'ar proved several striking results about the higher
direct images of dualizing sheaves: when $f \colon X \to Y$ is a surjective morphism
from a smooth projective variety $X$ to an arbitrary projective variety $Y$,
the sheaves $R^i \fl \omX$ are torsion-free and satisfy a vanishing theorem, and the
complex $\derR \fl \omX$ splits in the derived category $\Dbcoh(\OY)$. He also
conjectured that the same results should hold in much greater generality
\cite[Section~5]{KollarII}. More precisely, he predicted that, for every polarizable
variation of Hodge structure $V$ on a Zariski-open subset of a projective variety
$X$, there should be a coherent sheaf $S(X, V)$ on $X$ with the following
properties:%
\footnote{This is a simplified version; the actual conjecture is more precise.}
\begin{nenumerate}
\item \label{en:Kollar-1}
If $X$ is smooth and $V$ is defined on all of $X$, then $S(X, V) = \omX
\tensor F^p \shV$, where $F^{p+1} \shV = 0$ and $F^p \shV \neq 0$. 
\item \label{en:Kollar-2}
If $L$ is an ample line bundle on $X$, then $H^j \bigl( X, S(X, V) \tensor L
\bigr) = 0$ for $j > 0$.
\item  \label{en:Kollar-4}
If $f \colon X \to Y$ is a surjective morphism to another projective variety $Y$,
then the sheaves $R^i \fl S(X, V)$ are torsion-free.
\item \label{en:Kollar-3}
In the same situation, one has a decomposition
\[
	\derR \fl S(X, V) \simeq \bigoplus_{i=0}^r 
		\bigl( R^i \fl S(X, V) \bigr) \decal{-i},
\]
where $r = \dim X - \dim Y$.
\end{nenumerate}

In fact, Koll\'ar's conjecture follows quite easily from the theory of mixed Hodge
modules. Saito announced the proof with the laconic remark: ``This implies
some conjecture by Koll\'ar, combined with the results in \S3.''
\cite[p.~276]{Saito-MHM}. Fortunately, he later published a more detailed proof \cite{Saito-K}.

The idea is to extend the polarizable variation of Hodge structure $V$ to a
polarizable Hodge module $M = (\Mmod, F_{\bullet} \Mmod, K) \in \HMZp{X}{X}{w}$ with
strict support $X$, by appealing to \theoremref{thm:structure}, and then to define 
\[
	S(X, V) = F_{p(M)} \Mmod,
\]
where $p(M) = \min \menge{p \in \ZZ}{F_p \Mmod \neq 0}$. Recall from
\exerciseref{ex:DR} that $F_{p(M)} \Mmod$ is a well-defined coherent sheaf on $X$,
even though $(\Mmod, F_{\bullet} \Mmod)$ only lives on some ambient projective space.
With this definition of $S(X, V)$, Koll\'ar's predictions become consequences of
general results about polarizable Hodge modules: the formula in \ref{en:Kollar-1}
follows from \theoremref{thm:structure}; the vanishing in \ref{en:Kollar-2} follows
from \theoremref{thm:vanishing}; the splitting of the complex $\derR \fl S(X, V)$
in \ref{en:Kollar-3} follows from \corollaryref{cor:decomposition-MF}; etc. Saito
also found a very pretty argument for showing that the higher direct
image sheaves $R^i \fl S(X, V)$ are torsion-free when $f \colon X \to Y$ is
surjective \cite[Proposition~2.6]{Saito-K}. It is based on the following observation.

\begin{proposition}
Let $f \colon X \to Y$ be a projective morphism between complex manifolds, and let
$M$ be a polarizable Hodge module with strict support $Z \subseteq X$. If $M'$ is any
summand of $\shH^i \fl M$, then $p(M') > p(M)$ unless $\Supp M' = f(Z)$. 
\end{proposition}

\begin{proof}
Here is a brief outline of the proof. Suppose that $\Supp M' \neq f(Z)$; then it
suffices to show that $F_{p(M)} \Mmod' = 0$. Using the compatibility of the nearby
cycle functor with direct images, one reduces the problem to showing that
\[
	F_{p(M)} \gr_0^V \Mmod_h = 0,
\]
where $V_{\bullet} \Mmod_h$ denotes the $V\!$-filtration with respect to $h = 0$, and
where $h = g \circ f$ for a locally defined holomorphic function $g$ that vanishes
along $\Supp M'$ but not along $f(Z)$. But because $F_{p(M)-1} \Mmod_h = 0$, this
follows from the fact that 
\[
	\partial_t \colon F_{p-1} \gr_{-1}^V \Mmod_h \to F_p \gr_0^V \Mmod_h
\]
is surjective for every $p \in \ZZ$ (see the discussion after
\definitionref{def:strict-spec}).
\end{proof}

\section{Various functors and strictness}

In this final chapter, we are concerned with various functors on mixed Hodge modules, in
particular, the direct image, duality, and inverse image functors. They are
relatively easy to define for perverse sheaves and regular holonomic $\Dmod$-modules,
but the presence of the filtration leads to complications. We shall also discuss some
of the remarkable properties of those filtered $\Dmod$-modules $(\Mmod, F_{\bullet}
\Mmod)$ that underlie mixed Hodge modules -- they are due to the strong restrictions on
the filtration $F_{\bullet} \Mmod$ that are built into the definition.

\subsection{Derived category and strictness}
\label{par:DF}

In this section, we briefly discuss the definition of the derived category of
filtered $\Dmod$-modules, and the important notion of strictness. 
Let $X$ be a complex manifold, and let $\MF(\Dmod_X)$ denote the category of filtered
$\Dmod$-modules; its objects are pairs $(\Mmod, F_{\bullet} \Mmod)$, where $\Mmod$ is
a right $\Dmod_X$-module, and $F_{\bullet} \Mmod$ is a compatible filtration by
$\OX$-submodules. Note that $\MF(\Dmod_X)$ is not an abelian category. It is,
however, an exact category: a sequence of the form
\[
\begin{tikzcd}[column sep=small]
	0 \rar & (\Mmod', F_{\bullet} \Mmod') \rar{u} & (\Mmod, F_{\bullet} \Mmod) 
		\rar{v} & (\Mmod'', F_{\bullet} \Mmod'') \rar & 0,
\end{tikzcd}
\]
is considered to be short exact if and only if $v \circ u = 0$ and 
\[
\begin{tikzcd}[column sep=2em]
	0 \rar & \gr_{\bullet}^F \Mmod' \rar{\gr u} & \gr_{\bullet}^F \Mmod 
		\rar{\gr v} & \gr_{\bullet}^F \Mmod'' \rar & 0,
\end{tikzcd}
\]
is a short exact sequence of graded $\gr_{\bullet}^F \! \Dmod_X$-modules. Acyclic
complexes are defined in a similar way; after localizing at the subcategory of all
acyclic complexes, one obtains the derived category $\DF(\Dmod_X)$. This category
turns out to have a natural t-structure, whose heart is an abelian category
containing $\MF(\Dmod_X)$; more details about this construction can be found in
\cite{Laumon}.

In fact, the category $\DF(\Dmod_X)$ and the t-structure on it can be described in
more concrete terms, starting from the general principle that graded objects form an
abelian category, whereas filtered objects do not. Consider the
\define{Rees algebra} of $(\Dmod_X, F_{\bullet} \Dmod_X)$, defined as
\[
	\Rmod_X = R_F \Dmod_X = \bigoplus_{k = 0}^{\infty} F_k \Dmod_X \cdot z^k;
\]
here $z$ is an auxiliary variable. It is a graded $\OX$-algebra with
\[
	\Rmod_X / (z-1) \Rmod_X \simeq \Dmod_X 
		\quad \text{and} \quad
	\Rmod_X / z \Rmod_X \simeq \gr_{\bullet}^F \! \Dmod_X,
\]
and therefore interpolates between the non-commutative $\OX$-algebra $\Dmod_X$ and its
associated graded $\gr_{\bullet}^F \! \Dmod_X \simeq \Sym^{\bullet} \shT_X$.
By the same method, we can associate with every filtered $\Dmod$-module $(\Mmod,
F_{\bullet} \Mmod)$ a graded $\Rmod_X$-module 
\[
	R_F \Mmod = \bigoplus_{k \in \ZZ} F_k \Mmod \cdot z^k
		\subseteq \Mmod \tensor_{\OX} \OX \lbrack z, z^{-1} \rbrack,
\] 
which is coherent over $\Rmod_X$ exactly when the filtration $F_{\bullet} \Mmod$ is
good. 

\begin{exercise}
Prove that $R_F \Mmod / (z-1) R_F \Mmod \simeq \Mmod$ and $R_F \Mmod / z R_F \Mmod
\simeq \gr_{\bullet}^F \! \Mmod$.
\end{exercise}

This construction defines a functor $R_F \colon \MF(\Dmod_X) \to \MG(\Rmod_X)$
from the category of filtered right $\Dmod$-modules to the category of graded right
$\Rmod$-modules. 

\begin{exercise}
Prove that $R_F \colon \MF(\Dmod_X) \to \MG(\Rmod_X)$ is faithful
and identifies $\MF(\Dmod_X)$ with the subcategory of all graded $\Rmod$-modules
without $z$-torsion.
\end{exercise}

One can show that $R_F$ induces an equivalence of categories
\begin{equation} \label{eq:RF}
	R_F \colon \DbcohF(\Dmod_X) \to \DbcohG(\Rmod_X),
\end{equation}
under which the t-structure on $\DbcohF(\Dmod_X)$ corresponds to the standard
t-structure on $\DbcohG(\Rmod_X)$. From this equivalence, we get a collection of
cohomology functors 
\[
	\shH^i \colon \DbcohF(\Dmod_X) \to \MGcoh(\Rmod_X);
\]
the thing to keep in mind is that $\shH^i$ of a complex of filtered $\Dmod$-modules
is generally \emph{not} a filtered $\Dmod$-module, only a graded $\Rmod$-module. By a
similar procedure, one can define the bounded derived category of filtered regular
holonomic $\Dmod$-modules with $\QQ$-structure; in that case, the cohomology functors
$\shH^i$ go to the abelian category of graded regular holonomic
$\Rmod$-modules with $\QQ$-structure.  

\begin{definition}
A graded $\Rmod$-module is called \define{strict} if it has no $z$-torsion.
More generally, a complex of filtered $\Dmod$-modules is called \define{strict} if,
for every $i \in \ZZ$, the graded $\Rmod$-module obtained by applying $\shH^i$ is
strict.
\end{definition}

In other words, a graded $\Rmod$-module comes from a filtered $\Dmod$-module if and
only if it is strict. Strictness of a complex is equivalent to all differentials in
the complex being strictly compatible with the filtrations.

\begin{exercise}
Show that a complex of filtered $\Dmod$-modules
\[
\begin{tikzcd}[column sep=small]
	(\Mmod', F_{\bullet} \Mmod') \rar{u} & (\Mmod, F_{\bullet} \Mmod) 
		\rar{v} & (\Mmod'', F_{\bullet} \Mmod'')
\end{tikzcd}
\]
is strict at $(\Mmod, F_{\bullet} \Mmod)$ if and only if $u$ is strictly compatible
with the filtrations.
\end{exercise}

The definition of various derived functors (such as the direct image functor or the
duality functor) requires an abelian category, and therefore produces not filtered
$\Dmod$-modules but graded $\Rmod$-modules. For Saito's theory, it is important to
know that they are strict in the case of mixed Hodge modules.

\subsection{Direct image functor}
\label{par:direct-image}

Our first concern is the definition of the direct image functor in the case 
where $f \colon X \to Y$ is a proper morphism between two complex manifolds; by using
local charts, one can extend the definition to the case where $X$ and $Y$ are
analytic spaces. We would like to have an exact functor
\[
	\fp \colon \DbcohF(\Dmod_X) \to \DbcohF(\Dmod_Y)
\]
that is compatible with the direct image functor $\derR \fl$ for constructible
complexes. Saito's construction is based on his theory of filtered differential
complexes \cite[\S2.2]{Saito-HM}; what we shall do here is briefly sketch another one
based on Koszul duality.

Recall first how the elementary construction of $\fp$ works in the case of a single
filtered right $\Dmod$-module $(\Mmod, F_{\bullet} \Mmod)$. Using the factorization
\[
\begin{tikzcd}
	X \rar[hook] \arrow[bend right=40]{rr}{f} & X \times Y \rar{p_2} & Y
\end{tikzcd}
\]
through the graph of $f$, one only has to define the direct image for closed
embeddings and for smooth projections. In the case of a closed embedding $i \colon X
\into Y$, the direct image is again a filtered $\Dmod$-module, given by
\[
	\ip(\Mmod, F_{\bullet} \Mmod) = 
		\il \bigl( \Mmod \tensor_{\Dmod_X} \Dmod_{X \to Y} \bigr),
\]
where $(\Dmod_{X \to Y}, F_{\bullet} \Dmod_{X \to Y}) = \OX \tensor_{i^{-1} \OY}
i^{-1}(\Dmod_Y, F_{\bullet} \Dmod_Y)$. Locally, the embedding is defined by 
holomorphic functions $t_1, \dotsc, t_r$; if we set $\partial_i = \partial/\partial
t_i$, then
\[
	\ip \Mmod \simeq \Mmod \lbrack \partial_1, \dotsc, \partial_r \rbrack,
\]
with filtration given by
\[
	F_k \bigl( \ip \Mmod \bigr) 
		= \sum_{a \in \NN^r} F_{k-(a_1 + \dotsb + a_r)} \Mmod 
			\tensor \partial_1^{a_1} \dotsb \partial_r^{a_r}.
\]
Strictness is clearly not an issue in this situation. In the case of a smooth
projection $p_2 \colon X \times Y \to Y$, the direct image is a complex of filtered
$\Dmod$-modules, given by
\[
	\derR p_{2\ast} \DR_{X \times Y/Y}(\Mmod, F_{\bullet} \Mmod).
\]
If we use the canonical Godement resolution to define $\derR p_{2\ast}$, the
filtration on the complex is given by the subcomplexes
\[
	\derR p_{2\ast} F_k \DR_{X \times Y/Y}(\Mmod, F_{\bullet} \Mmod);
\]
there is no reason why the resulting filtered complex should be strict.

\begin{example}
For a morphism $f \colon X \to \pt$ to a point, strictness of $\fp(\Mmod,
F_{\bullet} \Mmod)$ is equivalent to the $E_1$-degeneration of the Hodge-de Rham
spectral sequence
\[
	E_1^{p,q} = H^{p+q} \bigl( X, \gr_{-p}^F \DR(\Mmod) \bigr)
		\Longrightarrow H^{p+q} \bigl( X, \DR(\Mmod) \bigr).
\]
This degeneration is of course an important issue in classical Hodge theory, too.
\end{example}

A one-step construction of the direct image functor uses the equivalence of
categories in \eqref{eq:RF} and Koszul duality \cite[Theorem~2.12.6]{BGS}; it works
more or less in the same way as Saito's theory of induced $\Dmod$-modules
\cite{Saito-ind}. This construction also explains why direct images are more
naturally defined using right $\Dmod$-modules. 

The key point is that the graded $\OX$-algebra $\Rmod_X$ is a \define{Koszul
algebra}, meaning that a certain Koszul-type complex
constructed from it is exact. It therefore has a Koszul dual $\Kmod_X$, which
is a certain graded $\OX$-algebra constructed from the holomorphic de Rham
complex $(\OmX^{\bullet}, d)$; concretely, $\Kmod_X = \Omega_X^{\bullet} \lbrack d
\rbrack$, with relations $d^2 = 0$ and $\lbrack d, \alpha \rbrack + (-1)^{\deg \alpha} d
\alpha = 0$. Now \define{Koszul duality} gives an equivalence of categories of right
modules
\[
	K_X \colon \DbcohG(\Rmod_X) \to \DbcohG(\Kmod_X).
\]
What makes this useful is that one has a morphism $\Kmod_Y \to \fl \Kmod_X$, and so $\derR
\fl$ of a complex of graded right $\Kmod_X$-modules is naturally a complex of graded
right $\Kmod_Y$-modules. Since $\Kmod_X$ has finite rank as on $\OX$-module, and $f$
is proper, it is also obvious that this operation preserves boundedness and
coherence. If we put everything together, we obtain an exact functor
\[
	\fp \colon \DbcohG(\Rmod_X) \to \DbcohG(\Rmod_Y)
\]
by taking the composition $K_Y^{-1} \circ \derR \fl \circ K_X$.

\begin{example}
Consider the case of a filtered $\Dmod$-module $(\Mmod, F_{\bullet} \Mmod)$.
By definition, 
\[
	\fp(\Mmod, F_{\bullet} \Mmod) = \fp(R_F \Mmod) \in \DbcohG(\Rmod_Y),
\]
and without additional assumptions on the filtration $F_{\bullet} \Mmod$, the graded
$\Rmod_Y$-modules $\shH^i \fp(\Mmod, F_{\bullet} \Mmod)$ may fail to be strict.
\end{example}

\subsection{Strictness of direct images}

As part of \theoremref{thm:direct-image}, Saito proved that when $(\Mmod, F_{\bullet}
\Mmod)$ underlies a polarizable Hodge module on a complex manifold $X$, and $f \colon
X \to Y$ is a projective morphism to another complex manifold $Y$, then the complex
$\fp(\Mmod, F_{\bullet} \Mmod)$ is strict; the result is easily extended to mixed
Hodge modules and to proper morphisms between algebraic varieties.  In particular,
every cohomology module $\shH^i \fp(\Mmod, F_{\bullet} \Mmod)$ is again a filtered
$\Dmod$-module. 

Saito's theorem is a powerful generalization of the fact that, on a smooth projective
variety $X$, the Hodge-de Rham spectral sequence degenerates at $E_1$. It has
numerous important consequences; here we only have room to mention one, namely
\define{Laumon's formula} for the associated graded of the
filtered $\Dmod$-module underlying the Hodge module $\shH^i \fl M$.
Laumon \cite{Laumon-for} described the ``associated graded'' of the complex
$\fp(\Mmod, F_{\bullet} \Mmod) \in \DbcohF(\Dmod_Y)$, which really means the
derived tensor product with $\Rmod_X / z \Rmod_X$. This is not at all the
same thing as the associated graded of the underlying $\Dmod$-modules $\shH^i \fp
\Mmod$, except when the complex is strict. In the case where $M$ is a mixed
Hodge module and $f \colon X \to Y$ is a projective morphism, this leads to the
isomorphism
\[
	\gr_{\bullet}^F \bigl( \shH^i \fp \Mmod \bigr) \simeq
		R^i \fl \Bigl( \gr_{\bullet}^F \! \Mmod 
			\Ltensor_{\Amod_X} \fu \Amod_Y \Bigr),
\]
where $\Amod_X = \gr_{\bullet}^F \! \Dmod_X$. To say this in more geometric terms,
let $\shC(X, M)$ denote the coherent sheaf on the cotangent bundle $T^{\ast} X$
determined by $\gr_{\bullet}^F \! \Mmod$. Then
\[
	\shC(Y, \shH^i \fl M) 
		\simeq R^i p_{1\ast} \Bigl( \derL(\mathit{df})^{\ast} \shC(X, M) \Bigr),
\]
where the notation is as in the following diagram:
\[
\begin{tikzcd}
	T^{\ast} Y \times_Y X \dar{p_1} \rar{\mathit{df}} & T^{\ast} X \\
	T^{\ast} Y 
\end{tikzcd}
\]
Laumon's formula explains, in the case of mixed Hodge modules, how taking the
associated graded interacts with the direct image functor. Saito has proved a similar
statement for the associated graded of the de Rham complex \cite[\S2.3.7]{Saito-HM}.

\begin{theorem}
Let $f \colon X \to Y$ be a projective morphism between complex manifolds. If
$(\Mmod, F_{\bullet} \Mmod)$ underlies a mixed Hodge module on $X$, one has
\[
	\derR \fl \gr_p^F \DR(\Mmod) \simeq \gr_p^F \DR(\fp \Mmod)
		\simeq \bigoplus_{i \in \ZZ} \gr_p^F \DR \bigl( \shH^i \fp \Mmod \bigr) \decal{-i}
\]
for every $p \in \ZZ$.
\end{theorem}

\subsection{Duality functor}
\label{par:duality}

We also have to say a few words about the definition of the duality functor for mixed
Hodge modules. The goal is to have an exact functor
\[
	\DD \colon \DbcohF(\Dmod_X) \to \DbcohF(\Dmod_X)^{\opp}
\]
that is compatible with the Verdier dual $\DD K$ for constructible complexes. It is
again easiest to define such a functor in terms of graded right $\Rmod$-modules.  The
point is that the tensor product $\omX \tensor_{\OX} \Rmod_X$ has two commuting
structures of right $\Rmod_X$-module; when we apply
\[
	\derR \shHom_{\Rmod_X} \Bigl( \argbl, \omX \tensor_{\OX} \Rmod_X \Bigr)
		\decal{\dim X}
\]
to a complex of graded right $\Rmod_X$-modules, we therefore obtain another complex
of graded right $\Rmod_X$-modules. Note that if we tensor by $\Rmod_X/(z-1) \Rmod_X$,
this operation specializes to the usual duality functor
\[
	\derR \shHom_{\Dmod_X} \Bigl( \argbl, \omX \tensor_{\OX} \Dmod_X \Bigr)
		\decal{\dim X}
\]
for $\Dmod$-modules, and is therefore nicely compatible with Verdier duality.

\begin{example}
In the case of a single filtered regular holonomic $\Dmod$-module $(\Mmod, F_{\bullet}
\Mmod)$, the complex of graded $\Rmod_X$-modules
\begin{equation} \label{eq:dual-MF}
	\derR \shHom_{\Rmod_X} \Bigl( R_F \Mmod, \omX \tensor_{\OX} \Rmod_X \Bigr)
		\decal{\dim X}
\end{equation}
can have cohomology in negative degrees and therefore fail to be strict. In general,
all one can say is that the cohomology in negative degrees must be $z$-torsion,
because the complex reduces to the holonomic dual of $\Mmod$ after tensoring by
$\Rmod_X/(z-1)\Rmod_X$.
\end{example}

Fortunately, the problem above does not arise for Hodge modules; this is the content
of the following theorem by Saito \cite[Lemme~5.1.13]{Saito-HM}.

\begin{theorem} \label{thm:duality}
If $(\Mmod, F_{\bullet} \Mmod)$ underlies a Hodge module $M \in \HM{X}{w}$, then the
dual complex is strict and again underlies a Hodge module $\DD M \in \HM{X}{-w}$.
\end{theorem}

It is proved by induction on the dimension of $\Supp M$; the key point is the
compatibility of the duality functor with nearby and vanishing cycles. The result
extends without much trouble to the case of mixed Hodge modules: given a mixed Hodge
module $(M, W_{\bullet} M) \in \MHM{X}$, the pair
\[
	\bigl( \DD M, \DD W_{-\bullet} M \bigr)
\]
is again a mixed Hodge module.

In fact, the strictness of the complex in \eqref{eq:dual-MF} is equivalent to
$\gr_{\bullet}^F \Mmod$ being \define{Cohen-Macaulay} as an $\Amod_X$-module, where
$\Amod_X = \gr_{\bullet}^F \! \Dmod_X$; therefore \theoremref{thm:duality} is saying
that $\gr_{\bullet}^F \Mmod$ is a Cohen-Macaulay module whenever $(\Mmod, F_{\bullet}
\Mmod)$ underlies a mixed Hodge module.  This has the following useful consequence:
if $(\Mmod', F_{\bullet} \Mmod')$ denotes the filtered $\Dmod$-module underlying $M'
= \DD M$, then
\[
	\gr_{\bullet}^F \! \Mmod'
		\simeq \derR \shHom_{\Amod_X} \Bigl( \gr_{\bullet}^F \! \Mmod,
			\omX \tensor_{\OX} \Amod_X \Bigr) \decal{n},
\]
where $n = \dim X$. It also implies that the coherent sheaf $\shC(X, M)$ on the
cotangent bundle $T^{\ast} X$ is Cohen-Macaulay.  This fact can be used to get
information about the coherent sheaves $R^i \shHom_{\OX}(\gr_p^F \! \Mmod, \OX)$ from
the geometry of the characteristic variety $\Ch(\Mmod)$ of the $\Dmod$-module
\cite{mhmduality}. For example, suppose that the fiber of the projection $\Ch(\Mmod)
\to X$ over a point $x \in X$ has dimension $\leq d$; then
\[
	R^i \shHom_{\OX}(\gr_p^F \! \Mmod, \OX) = 0 \quad \text{for every $i \geq d + 1$.}
\]

\subsection{Inverse image functors}
\label{par:inverse-image}

The inverse image functors for mixed Hodge modules lift the two operations $f^{-1}
K$ and $\fus K$ for perverse sheaves. We shall only discuss the case
of a morphism $f \colon Y \to X$ between two complex manifolds; some additional work
is needed to deal with morphisms between analytic spaces. Using the factorization
\[
\begin{tikzcd}
	Y \rar[hook] \arrow[bend right=40]{rr}{f} & Y \times X \rar{p_2} & X
\end{tikzcd}
\]
through the graph of $f$, it is enough to define inverse images for closed embeddings
and for smooth projections.

Let us first consider the case of a smooth morphism $f \colon Y \to X$. For
variations of Hodge structure, the inverse image is obtained by pulling back the
vector bundle and the connection; to get a sensible definition for Hodge
modules, we therefore have to tensor by the relative canonical bundle $\omYX = \omY
\tensor \fu \omX^{-1}$. There also has to be a shift by the relative dimension $r =
\dim Y - \dim X$ of the morphism, because of Saito's convention that a variation of
Hodge structure of weight $w$ defines a Hodge module of weight $w + \dim X$.

With this in mind, let $M = (\Mmod, F_{\bullet} \Mmod, K)$ be
a filtered regular holonomic $\Dmod$-module with $\QQ$-structure on $X$. We first
define an auxiliary object
\[
	\Mt = \bigl( \Mmodt, F_{\bullet} \Mmodt, \Kt \bigr),
\]
where $\Kt = f^{-1} K(-r)$ is again a perverse sheaf on $Y$, and where
\[
	\Mmodt = \omYX \tensor_{\OY} \fu \Mmod \quad \text{and} \quad
		F_p \Mmodt = \omYX \tensor_{\OY} \fu F_{p+r} \Mmod
\]
is again a filtered $\Dmod$-module on $Y$, through the natural morphism
$\Dmod_Y \to \fu \Dmod_X$. It requires a small calculation to show that
$\DR(\Mmodt) \simeq \CC \tensor_{\QQ} \Kt$; this is where the factor $\omYX$ comes
in. Now the crucial result is the following.

\begin{theorem} \label{thm:pullback}
Suppose that $M \in \HMp{X}{w}$. Then $\Mt \in \HMp{Y}{w+r}$.
\end{theorem}

Unfortunately, one cannot prove directly that $\Mt$ satisfies the conditions for
being a Hodge module; instead, the proof has to go through the equivalence of categories
in \theoremref{thm:structure}. Since $M$ admits a decomposition by strict support, we
may assume that $M \in \HMZp{Z}{X}{w}$, and therefore comes from a generically
defined polarizable variation of Hodge structure on $Z$ of weight $w - \dim Z$; its
pullback is a generically defined polarizable variation of Hodge structure on
$f^{-1}(Z)$ of the same weight, and therefore extends uniquely to an object of
\[
	\HMZp{f^{-1}(Z)}{Y}{w+r}. 
\]
One then checks that this extension is isomorphic to $\Mt$.

With this surprisingly deep result in hand, it is straightforward to construct
the two inverse image functors in general. Let $(M, W_{\bullet}) \in \MHW{X}$ be a
weakly mixed Hodge module on $X$. For a smooth morphism $f \colon Y \to X$, we define
\begin{align*}
	\shH^{-r} \fu(M, W_{\bullet} M) &= \Bigl( \Mt, W_{\bullet+r} \Mt \Bigr) \\
	\shH^r \fus(M, W_{\bullet} M) &= \Bigl( \Mt(r), W_{\bullet-r} \Mt(r) \Bigr)
\end{align*}
as objects of $\MHW{Y}$; the need for shifting the weight filtration is explained
by \theoremref{thm:pullback}. One can show that both functors take
(graded-polarizable) mixed Hodge modules to (graded-polarizable) mixed Hodge modules.

Next, we turn our attention to the case of a closed embedding $i \colon Y \into X$.
Here the idea is to construct cohomological functors
\begin{equation} \label{eq:iu-ius}
	\shH^j \iu \colon \MHM{X} \to \MHM{Y}
		\quad \text{and} \quad
	\shH^j \ius \colon \MHM{X} \to \MHM{Y}
\end{equation}
in terms of nearby and vanishing cycles. This procedure has the advantage of
directly producing mixed Hodge modules; note that it leads in general to a
different filtration than the one used by Laumon for arbitrary filtered $\Dmod$-modules
\cite{Laumon}. 

To make things simpler, let me explain the construction when $\dim Y = \dim X -
1$. Locally, the submanifold $Y$ is then defined by a single holomorphic function
$t$, so let us first treat the case where $Y = t^{-1}(0)$. Fix a mixed Hodge module $M \in
\MHM{X}$ and set $K = \rat M$. By construction of the vanishing cycle functor (see
\parref{par:review-cycles}), we have a distinguished triangle
\[
	i^{-1} K \to \psi_{t,1} K \to \phi_{t,1} K \to i^{-1} K \decal{1}
\]
in the derived category of constructible complexes; it is therefore reasonable to
define $\shH^j \iu M \in \MHM{Y}$, for $j \in \{-1, 0\}$, as the cohomology modules
of the complex of mixed Hodge modules
\[
\begin{tikzcd}[column sep=scriptsize]
	\Bigl\lbrack \psi_{t,1} M \rar{\can} & \phi_{t,1} M \Bigr\rbrack \decal{1}.
\end{tikzcd}
\]
To describe this operation on the level of filtered $\Dmod$-modules, let $V_{\bullet}
\Mmod$ denote the $V\!$-filtration relative to the divisor $t = 0$; then the
corresponding complex of filtered $\Dmod$-modules is
\begin{equation} \label{eq:iu-MF}
\begin{tikzcd}[column sep=scriptsize]
	\Bigl\lbrack 
		\bigl( \gr_{-1}^V \Mmod, F_{\bullet-1} \gr_{-1}^V \Mmod \bigr) \rar{\partial_t} &
		\bigl( \gr_0^V \Mmod, F_{\bullet} \gr_0^V \Mmod \bigr) \Bigr\rbrack \decal{1},
\end{tikzcd}
\end{equation}
where the filtration is the one induced by $F_{\bullet} \Mmod$. 

To deal with the general case, we observe that the $V\!$-filtration is independent of
the choice of local equation for $Y$, and that both $\gr_{-1}^V \Mmod$ and $\gr_0^V
\Mmod$ are well-defined sheaves of $\OY$-modules. The same is true for the action of
$t \partial_t$; the only thing that actually depends on $t$ is the $\Dmod_Y$-module
structure. In fact, both sheaves carry an action by $\gr_0^V
\Dmod_X$, but this $\OY$-algebra is only locally isomorphic to $\Dmod_Y \lbrack t
\partial_t \rbrack$. On the other hand, one can show that the quotient of $\gr_0^V
\Dmod_X$ by $t \partial_t$ is canonically isomorphic to $\Dmod_Y$; this gives both
\[
	\coker \Bigl( \partial_t \colon \gr_{-1}^V \Mmod \to \gr_0^V \Mmod \Bigr)
		\quad \text{and} \quad
	\ker \Bigl( \partial_t \colon \gr_{-1}^V \Mmod
		\to \gr_0^V \Mmod \Bigr) \tensor_{\OY} \mathscr{N}_{Y \mid X}
\]
a canonical $\Dmod_Y$-module structure that does not depend on the choice of $t$.
Thus we obtain two filtered regular holonomic $\Dmod$-modules with $\QQ$-structure
and weight filtration; because the conditions in the definition are local, both are
mixed Hodge modules on $Y$. 

The same procedure can be used to define $\shH^j \ius M \in \MHM{Y}$, for $j \in \{0,
1\}$. When $Y = t^{-1}(0)$, one has a similar distinguished triangle for $\ius K$,
which suggests to look at the complex of mixed Hodge modules
\[
\begin{tikzcd}[column sep=scriptsize]
	\Bigl\lbrack \phi_{t,1} M \rar{\var} & \psi_{t,1} M(-1) \Bigr\rbrack.
\end{tikzcd}
\]
The corresponding complex of filtered $\Dmod$-modules is
\begin{equation} \label{eq:ius-MF}
\begin{tikzcd}[column sep=scriptsize]
	\Bigl\lbrack 
		\bigl( \gr_0^V \Mmod, F_{\bullet} \gr_0^V \Mmod \bigr) \rar{t} &
		\bigl( \gr_{-1}^V \Mmod, F_{\bullet} \gr_{-1}^V \Mmod \bigr) \Bigr\rbrack.
\end{tikzcd}
\end{equation}
As before, one proves that the cohomology of this complex leads to two
well-defined mixed Hodge modules on $Y$. 

\begin{exercise}
Let $\ius \Mmod$ denote the complex of $\Dmod$-modules in \eqref{eq:ius-MF}. Use the
fact that $(\Mmod, F_{\bullet} \Mmod)$ is quasi-unipotent and regular along $t = 0$ to
construct a morphism
\[
	F_p(\ius \Mmod) \to \derL \iu(F_p \Mmod) \decal{-1}
\]
in the derived category $\Dbcoh(\OX)$. Show that this morphism is an isomorphism when
$\Mmod$ comes from a vector bundle with integrable connection.
\end{exercise}

For an arbitrary morphism $f \colon Y \to X$, we use
the factorization $f = p_2 \circ i$ given by the graph of $f$, and define
\begin{align*}
	\shH^j \fu(M, W_{\bullet} M) &= \shH^{j + \dim Y} \iu 
		\Bigl( \shH^{-\dim Y} \pu_2(M, W_{\bullet} M) \Bigr) \\
	\shH^j \fus(M, W_{\bullet} M) &= \shH^{j - \dim Y} \ius 
		\Bigl( \shH^{\dim Y} \pus_2(M, W_{\bullet} M) \Bigr).
\end{align*}
One can check that the resulting functors take $\MHMp{X}$ into $\MHMp{Y}$, and that
they are (up to canonical isomorphism) compatible with composition.


\bibliographystyle{amsalphax}

\newcommand{\etalchar}[1]{$^{#1}$}
\providecommand{\bysame}{\leavevmode\hbox to3em{\hrulefill}\thinspace}
\providecommand{\ZM}{\relax\ifhmode\unskip\space\fi Zbl }
\providecommand{\MR}{\relax\ifhmode\unskip\space\fi MR }
\providecommand{\arXiv}[1]{\relax\ifhmode\unskip\space\fi\href{http://arxiv.org/abs/#1}{arXiv:#1}}
\providecommand{\MRhref}[2]{%
  \href{http://www.ams.org/mathscinet-getitem?mr=#1}{#2}
}
\providecommand{\href}[2]{#2}

\end{document}

%% file: macros-bib.tex
\usepackage{amsmath,amsfonts,amsthm,mathrsfs}
\usepackage{amssymb}

\usepackage[ocgcolorlinks,unicode,bookmarks]{hyperref}
\usepackage[usenames,dvipsnames]{xcolor}
\hypersetup{colorlinks=true,citecolor=NavyBlue,linkcolor=BrickRed,urlcolor=Green} 

\usepackage{expl3}

\ExplSyntaxOn
\prop_new:N \g_cite_map_prop
\tl_new:N \l_citekey_result_tl

\cs_new:Npn \mapcitekey #1#2 {
  \clist_map_inline:nn {#2}
       {  \prop_gput:Nnn  \g_cite_map_prop  {##1} {#1}   }
}

\cs_new:Npn \getcitekey #1 {
   \prop_get:NoN \g_cite_map_prop{#1}  \l_citekey_result_tl
   \quark_if_no_value:NF \l_citekey_result_tl
       {  \tl_set_eq:NN #1  \l_citekey_result_tl  }
}

\cs_new:Npn \showcitekeymaps {\prop_show:N  \g_cite_map_prop }
\ExplSyntaxOff

\usepackage{etoolbox}
\makeatletter
\patchcmd{\@citex}{\if@filesw}{\getcitekey\@citeb \if@filesw}%
    {\typeout{*** SUCCESS ***}}{\typeout{*** FAIL ***}}
\patchcmd{\nocite}{\if@filesw}{\getcitekey\@citeb \if@filesw}%
    {\typeout{*** SUCCESS ***}}{\typeout{*** FAIL ***}}
\makeatother

\usepackage{enumitem}

\newenvironment{nenumerate}{%
	\begin{enumerate}[label=(\arabic{*}), ref=(\arabic{*})]
}{%
	\end{enumerate}%
}
\newenvironment{aenumerate}{%
	\begin{enumerate}[label=(\alph{*}), ref=(\alph{*})]
}{%
	\end{enumerate}%
}

\usepackage{chngcntr}

\ifdraft
\usepackage[notcite,notref,color]{showkeys}

\definecolor{labelkey}{gray}{0.5}
\fi

\usepackage{tikz}

\usepackage{tikz-cd}
\tikzset{commutative diagrams/arrow style=math font}

\usetikzlibrary{matrix,arrows}
\newlength{\myarrowsize} 

\pgfarrowsdeclare{cmto}{cmto}{
	\pgfsetdash{}{0pt} 
	\pgfsetbeveljoin 
	\pgfsetroundcap 
	\setlength{\myarrowsize}{0.6pt}
	\addtolength{\myarrowsize}{.5\pgflinewidth}
	\pgfarrowsleftextend{-4\myarrowsize-.5\pgflinewidth} 
	\pgfarrowsrightextend{.8\pgflinewidth}
}{
	\setlength{\myarrowsize}{0.6pt} 
  	\addtolength{\myarrowsize}{.5\pgflinewidth}  
	\pgfsetlinewidth{0.5\pgflinewidth}
	\pgfsetroundjoin
	\pgfpathmoveto{\pgfpoint{1.5\pgflinewidth}{0}}
	\pgfpatharc{-109}{-170}{4\myarrowsize}
	\pgfpatharc{10}{189}{0.58\pgflinewidth and 0.2\pgflinewidth}
	\pgfpatharc{-170}{-115}{4\myarrowsize+\pgflinewidth}
	\pgfpathclose
	\pgfusepathqfillstroke
	\pgfpathmoveto{\pgfpoint{1.5\pgflinewidth}{0}}
	\pgfpatharc{109}{170}{4\myarrowsize}
	\pgfpatharc{-10}{-189}{0.58\pgflinewidth and 0.2\pgflinewidth}
	\pgfpatharc{170}{115}{4\myarrowsize+\pgflinewidth}
	\pgfpathclose
	\pgfusepathqfillstroke
	\pgfsetlinewidth{2\pgflinewidth}
}

\pgfarrowsdeclare{cmonto}{cmonto}{
	\pgfsetdash{}{0pt} 
	\pgfsetbeveljoin 
	\pgfsetroundcap 
	\setlength{\myarrowsize}{0.6pt}
	\addtolength{\myarrowsize}{.5\pgflinewidth}
	\pgfarrowsleftextend{-4\myarrowsize-.5\pgflinewidth} 
	\pgfarrowsrightextend{.8\pgflinewidth}
}{
	\setlength{\myarrowsize}{0.6pt} 
  	\addtolength{\myarrowsize}{.5\pgflinewidth}  
	\pgfsetlinewidth{0.5\pgflinewidth}
	\pgfsetroundjoin
	\pgfpathmoveto{\pgfpoint{1.5\pgflinewidth}{0}}
	\pgfpatharc{-109}{-170}{4\myarrowsize}
	\pgfpatharc{10}{189}{0.58\pgflinewidth and 0.2\pgflinewidth}
	\pgfpatharc{-170}{-115}{4\myarrowsize+\pgflinewidth}
	\pgfpathclose
	\pgfusepathqfillstroke
	\pgfpathmoveto{\pgfpoint{1.5\pgflinewidth}{0}}
	\pgfpatharc{109}{170}{4\myarrowsize}
	\pgfpatharc{-10}{-189}{0.58\pgflinewidth and 0.2\pgflinewidth}
	\pgfpatharc{170}{115}{4\myarrowsize+\pgflinewidth}
	\pgfpathclose
	\pgfusepathqfillstroke
	\pgfpathmoveto{\pgfpoint{1.5\pgflinewidth-0.3em}{0}}
	\pgfpatharc{-109}{-170}{4\myarrowsize}
	\pgfpatharc{10}{189}{0.58\pgflinewidth and 0.2\pgflinewidth}
	\pgfpatharc{-170}{-115}{4\myarrowsize+\pgflinewidth}
	\pgfpathclose
	\pgfusepathqfillstroke
	\pgfpathmoveto{\pgfpoint{1.5\pgflinewidth-0.3em}{0}}
	\pgfpatharc{109}{170}{4\myarrowsize}
	\pgfpatharc{-10}{-189}{0.58\pgflinewidth and 0.2\pgflinewidth}
	\pgfpatharc{170}{115}{4\myarrowsize+\pgflinewidth}
	\pgfpathclose
	\pgfusepathqfillstroke
	\pgfsetlinewidth{2\pgflinewidth}
}

\pgfarrowsdeclare{cmhook}{cmhook}{
	\pgfsetdash{}{0pt} 
	\pgfsetbeveljoin 
	\pgfsetroundcap 
	\setlength{\myarrowsize}{0.6pt}
	\addtolength{\myarrowsize}{.5\pgflinewidth}
	\pgfarrowsleftextend{-4\myarrowsize-.5\pgflinewidth} 
	\pgfarrowsrightextend{.8\pgflinewidth}
}{
	\setlength{\myarrowsize}{0.6pt} 
  	\addtolength{\myarrowsize}{.5\pgflinewidth}  
 	\pgfsetdash{}{0pt}
	\pgfsetroundcap
	\pgfpathmoveto{\pgfqpoint{0pt}{-4.667\pgflinewidth}}
	\pgfpathcurveto
    {\pgfqpoint{4\pgflinewidth}{-4.667\pgflinewidth}}
    {\pgfqpoint{4\pgflinewidth}{0pt}}
    {\pgfpointorigin}
	\pgfusepathqstroke
}

\newenvironment{diagram*}[2]{%
\[%
\begin{tikzpicture}[>=cmto,baseline=(current bounding box.center),%
	to/.style={->,font=\scriptsize,cap=round},%
	into/.style={cmhook->,font=\scriptsize,cap=round},%
	onto/.style={-cmonto,font=\scriptsize,cap=round},%
	math/.style={matrix of math nodes, row sep=#2, column sep=#1,%
		text height=1.5ex, text depth=0.25ex}]%
}{%
\end{tikzpicture}%
\]%
\ignorespacesafterend%
}

\newcommand{\MHM}{\operatorname{MHM}}
\newcommand{\MHS}{\operatorname{MHS}}

\newcommand{\pt}{\mathit{pt}}
\newcommand{\Dmod}{\mathscr{D}}
\newcommand{\Mmod}{\mathcal{M}}

\newcommand{\shT}{\mathscr{T}}

\newcommand{\derR}{\mathbf{R}}
\newcommand{\derL}{\mathbf{L}}

\newcommand{\Ltensor}{\overset{\derL}{\tensor}}
\newcommand{\decal}[1]{\lbrack #1 \rbrack}

\newcommand{\shH}{\mathcal{H}}

\newcommand{\tensor}{\otimes}

\newcommand{\shHom}{\mathcal{H}\hspace{-1pt}\mathit{om}}

\newcommand{\NN}{\mathbb{N}}
\newcommand{\ZZ}{\mathbb{Z}}
\newcommand{\QQ}{\mathbb{Q}}
\newcommand{\RR}{\mathbb{R}}
\newcommand{\CC}{\mathbb{C}}
\newcommand{\HH}{\mathbb{H}}
\newcommand{\PP}{\mathbb{P}}

\newcommand{\menge}[2]{\bigl\{ \thinspace #1 \thinspace\thinspace \big\vert%
\thinspace\thinspace #2 \thinspace \bigr\}}

\DeclareMathOperator{\im}{im}

\DeclareMathOperator{\coker}{coker}

\DeclareMathOperator{\Res}{Res}

\DeclareMathOperator{\id}{id}

\renewcommand{\Re}{\operatorname{Re}}

\DeclareMathOperator{\Supp}{Supp}

\DeclareMathOperator{\rat}{rat}
\DeclareMathOperator{\Cone}{Cone}
\DeclareMathOperator{\Sym}{Sym}
\DeclareMathOperator{\gr}{gr}
\DeclareMathOperator{\DR}{DR}

\DeclareMathOperator{\Ext}{Ext}
\DeclareMathOperator{\GL}{GL}
\DeclareMathOperator{\SL}{SL}
\DeclareMathOperator{\can}{can}

\DeclareMathOperator{\Ch}{Ch}

\newcommand{\define}[1]{\emph{#1}}

\newcommand{\lie}[2]{\lbrack #1, #2 \rbrack}

\newcommand{\shf}[1]{\mathscr{#1}}
\newcommand{\OX}{\shf{O}_X}
\newcommand{\OmX}{\Omega_X}

\newcommand{\shV}{\shf{V}}

\newcommand{\restr}[1]{\big\vert_{#1}}

\newcommand{\argbl}{-}

\def\overbar#1#2#3{{%
	\setbox0=\hbox{\displaystyle{#1}}%
	\dimen0=\wd0
	\advance\dimen0 by -#2 
	\vbox {\nointerlineskip \moveright #3 \vbox{\hrule height 0.3pt width \dimen0}%
		\nointerlineskip \vskip 1.5pt \box0}%
}}

\newcommand{\dst}{\Delta^{\ast}}

\newcommand{\into}{\hookrightarrow}

\newcommand{\pil}{\pi_{\ast}}
\newcommand{\jl}{j_{\ast}}

\newcommand{\jls}{j_{!}}

\newcommand{\ju}{j^{\ast}}
\newcommand{\fu}{f^{\ast}}
\newcommand{\fl}{f_{\ast}}
\newcommand{\fus}{f^{!}}
\newcommand{\iu}{i^{\ast}}
\newcommand{\ius}{i^{!}}
\newcommand{\il}{i_{\ast}}
\newcommand{\pu}{p^{\ast}}
\newcommand{\pl}{p_{\ast}}

\newcommand{\tl}{t_{\ast}}

\newcommand{\piu}{\pi^{\ast}}

\newcommand{\fp}{f_{+}}

\newcommand{\DD}{\mathbb{D}}

\newcommand{\shO}{\shf{O}}

\makeatletter
\let\@@seccntformat\@seccntformat
\renewcommand*{\@seccntformat}[1]{%
  \expandafter\ifx\csname @seccntformat@#1\endcsname\relax
    \expandafter\@@seccntformat
  \else
    \expandafter
      \csname @seccntformat@#1\expandafter\endcsname
  \fi
    {#1}%
}
\newcommand*{\@seccntformat@subsection}[1]{%
  \textbf{\csname the#1\endcsname.}
}
\makeatother

\makeatletter
\let\@paragraph\paragraph
\renewcommand*{\paragraph}[1]{%
	\vspace{0.3\baselineskip}%
	\@paragraph{\textit{#1}}%
}
\makeatother

\counterwithin{equation}{subsection}
\counterwithout{subsection}{section}
\counterwithin{figure}{subsection}

\newtheorem{theorem}[equation]{Theorem}
\newtheorem*{theorem*}{Theorem}

\newtheorem*{lemma*}{Lemma}
\newtheorem{corollary}[equation]{Corollary}
\newtheorem{proposition}[equation]{Proposition}
\newtheorem*{proposition*}{Proposition}

\newtheorem*{conjecture*}{Conjecture}

\theoremstyle{definition}
\newtheorem{definition}[equation]{Definition}
\newtheorem*{definition*}{Definition}
\theoremstyle{remark}

\newtheorem{exercise}[equation]{Exercise}
\newtheorem{example}[equation]{Example}
\newtheorem*{example*}{Example}
\newtheorem*{problem*}{Problem}
\newtheorem*{note}{Note}

\theoremstyle{plain}

\newcommand{\theoremref}[1]{\hyperref[#1]{Theorem~\ref*{#1}}}
\newcommand{\lemmaref}[1]{\hyperref[#1]{Lemma~\ref*{#1}}}
\newcommand{\definitionref}[1]{\hyperref[#1]{Definition~\ref*{#1}}}
\newcommand{\propositionref}[1]{\hyperref[#1]{Proposition~\ref*{#1}}}
\newcommand{\conjectureref}[1]{\hyperref[#1]{Conjecture~\ref*{#1}}}
\newcommand{\corollaryref}[1]{\hyperref[#1]{Corollary~\ref*{#1}}}
\newcommand{\exampleref}[1]{\hyperref[#1]{Example~\ref*{#1}}}
\newcommand{\exerciseref}[1]{\hyperref[#1]{Exercise~\ref*{#1}}}

\makeatletter
\let\old@caption\caption
\renewcommand*{\caption}[1]{%
	\setcounter{figure}{\value{equation}}%
	\stepcounter{equation}%
	\old@caption{#1}\relax%
}
\makeatother

\newcounter{intro}

\newtheorem{intro-conjecture}[intro]{Conjecture}
\newtheorem{intro-corollary}[intro]{Corollary}
\newtheorem{intro-theorem}[intro]{Theorem}

\newcommand{\omY}{\omega_Y}

\newcommand{\OY}{\shO_Y}

\newcommand{\parref}[1]{\hyperref[#1]{\S\ref*{#1}}}
\newcommand{\chapref}[1]{\hyperref[#1]{Chapter~\ref*{#1}}}

\makeatletter
\newcommand*\if@single[3]{%
  \setbox0\hbox{${\mathaccent"0362{#1}}^H$}%
  \setbox2\hbox{${\mathaccent"0362{\kern0pt#1}}^H$}%
  \ifdim\ht0=\ht2 #3\else #2\fi
  }
\newcommand*\rel@kern[1]{\kern#1\dimexpr\macc@kerna}
\newcommand*\widebar[1]{\@ifnextchar^{{\wide@bar{#1}{0}}}{\wide@bar{#1}{1}}}
\newcommand*\wide@bar[2]{\if@single{#1}{\wide@bar@{#1}{#2}{1}}{\wide@bar@{#1}{#2}{2}}}
\newcommand*\wide@bar@[3]{%
  \begingroup
  \def\mathaccent##1##2{%
    \if#32 \let\macc@nucleus\first@char \fi
    \setbox\z@\hbox{$\macc@style{\macc@nucleus}_{}$}%
    \setbox\tw@\hbox{$\macc@style{\macc@nucleus}{}_{}$}%
    \dimen@\wd\tw@
    \advance\dimen@-\wd\z@
    \divide\dimen@ 3
    \@tempdima\wd\tw@
    \advance\@tempdima-\scriptspace
    \divide\@tempdima 10
    \advance\dimen@-\@tempdima
    \ifdim\dimen@>\z@ \dimen@0pt\fi
    \rel@kern{0.6}\kern-\dimen@
    \if#31
      \overline{\rel@kern{-0.6}\kern\dimen@\macc@nucleus\rel@kern{0.4}\kern\dimen@}%
      \advance\dimen@0.4\dimexpr\macc@kerna
      \let\final@kern#2%
      \ifdim\dimen@<\z@ \let\final@kern1\fi
      \if\final@kern1 \kern-\dimen@\fi
    \else
      \overline{\rel@kern{-0.6}\kern\dimen@#1}%
    \fi
  }%
  \macc@depth\@ne
  \let\math@bgroup\@empty \let\math@egroup\macc@set@skewchar
  \mathsurround\z@ \frozen@everymath{\mathgroup\macc@group\relax}%
  \macc@set@skewchar\relax
  \let\mathaccentV\macc@nested@a
  \if#31
    \macc@nested@a\relax111{#1}%
  \else
    \def\gobble@till@marker##1\endmarker{}%
    \futurelet\first@char\gobble@till@marker#1\endmarker
    \ifcat\noexpand\first@char A\else
      \def\first@char{}%
    \fi
    \macc@nested@a\relax111{\first@char}%
  \fi
  \endgroup
}
\makeatother